\documentclass[twoside, 12pt]{article}
\usepackage{amsmath,amssymb,amsthm,mathrsfs}
\usepackage[margin=2.15cm]{geometry} % showframe,
\usepackage{lipsum}
\usepackage{titlesec,hyperref}
\usepackage{fancyhdr}
\usepackage[numbers,sort&compress]{natbib}
\usepackage{color}

\fancypagestyle{plain}
{
\fancyhf{}

}

\titleformat{\subsection}{\it}{\thesubsection.\enspace}{1.5pt}{}
\titleformat{\subsubsection}{\it}{\thesubsubsection.\enspace}{1.5pt}{}

\newtheorem{theo}{Theorem}[section]
\newtheorem{lemm}[theo]{Lemma}

\newtheorem{rema}{Remark}[section]

\numberwithin{equation}{section}

\allowdisplaybreaks

\def\th2{\frac{\theta}{2}}

\begin{document}
\title{ Optimal Decay Rates of Classical Solutions for the Full Compressible MHD Equations\hspace{-4mm}}
\author{Jincheng Gao$^\dag$  \quad Qiang Tao $^\ddag$ \quad Zheng-an Yao $^\dag$\\[10pt]
\small {$^\dag $School of Mathematics and Computational Science, Sun Yat-sen University,}\\
\small {510275, Guangzhou, P. R. China}\\[5pt]
\small {$^\ddag$ College of Mathematics and Computational Science, Shenzhen University,}\\
\small {518060, Shenzhen, P. R. China}\\[5pt]
}

%\address{Department of Mathematics and Statistics, Indian Institute of Technology Kanpur, \\Kanpur, Uttar Pradesh, India\\}

\footnotetext{Email: \it gaojc1998@163.com(J.C.Gao), taoq060@126.com(Q.Tao), \it mcsyao@mail.sysu.edu.cn(Z.A.Yao).}

\date{}

\maketitle

\begin{abstract}
In this paper, we are concerned with optimal decay rates for higher order
spatial derivatives of classical solutions to the full compressible  MHD
equations in three dimensional whole space.
If the initial perturbation are small in $H^3$-norm and bounded in
$L^q(q\in \left[1, \frac{6}{5}\right))$-norm, we apply the Fourier
splitting method by Schonbek [Arch.Rational Mech. Anal. 88 (1985)]
to establish optimal decay rates for the second order spatial derivatives
of solutions and the third order spatial derivatives of magnetic field
in $L^2$-norm.
These results improve the work of Pu and Guo [Z. Angew. Math. Phys. 64 (2013) 519-538].

\vspace*{5pt}

\vspace*{5pt}
\end{abstract}

%\tableofcontents

\section{Introduction}
\quad In this paper, we are concerned with the compressible viscous and
heat-conductive magnetohydrodynamic (In short, MHD) equations
in the Eulerian coordinates
\begin{equation}\label{1.1}
\left\{
\begin{aligned}
&\rho_t+{\rm div}(\rho u)=0,\\
&(\rho u)_t+{\rm div}(\rho u\otimes u)-\mu \Delta u-(\mu+\lambda)\nabla {\rm div}u
  +\nabla P(\rho, \theta)=({\rm curl} B)\times B,\\
&c_\nu [(\rho \theta)_t+{\rm div}(\rho u \theta)]-\kappa\Delta \theta+\theta \partial_\theta P(\rho, \theta){\rm div}u
=2\mu|D(u)|^2+\lambda({\rm div}u)^2+\nu |{\rm curl} B|^2,\\
&B_t -{\rm curl}(u \times B)=\nu \Delta B, \ {\rm div} B=0,
\end{aligned}
\right.
\end{equation}
where $(x, t)\in \mathbb{R}^3 \times \mathbb{R}^+$.
Here the unknown functions $\rho, u=(u_1, u_2, u_3)^{tr}$, $\theta$ and
$B=(B_1, B_2, B_3)^{tr}$ represent the fluid density,
velocity, absolute temperature, and magnetic field respectively;
$D(u)$ is the deformation tensor and defined by
$$
D(u)=\frac{1}{2}\left[\nabla u+(\nabla u)^{tr}\right].
$$
The pressure function  $P(\rho, \theta)$ is smooth
and satisfies $P_{\rho}(1, 1)>0$ and $P_{\theta}(1, 1)>0$
in a neighborhood of $(1, 1)$.
The constants $\mu $ and $\lambda$ are the first and second viscosity
coefficients respectively and satisfy the physical restrictions
$$
\mu>0, \quad 2\mu+3\lambda \ge 0.
$$
Positive constants $c_v, \kappa$, and $\nu$ are respectively the heat
capacity, the ratio of the heat conductivity coefficient over the heat capacity,
and the magnetic diffusivity acting as a magnetic diffusion coefficient of
the magnetic field. For the sake of simplicity, we assume $c_\nu, P_\rho(1,1)$
and $P_\theta(1,1)$ to be $1$.
To complete the system \eqref{1.1}, the initial data are given by
\begin{equation}\label{1.2}
\left.(\rho, u, \theta, B)(x,t)\right|_{t=0}=(\rho_0(x), u_0(x), \theta_0(x), B_0(x)).
\end{equation}
Furthermore, as the spatial variable tends to infinity, we assume
\begin{equation}\label{1.3}
\underset{|x|\rightarrow \infty}{\lim}(\rho_0-1, u_0, \theta_0-1, B_0)(x)=0.
\end{equation}

The compressible MHD systems are combination of the compressible Navier-
Stokes equations of fluid dynamics and Maxwell's equations of electromagnetism.
On the other hand, although the electric field $E$ does not apper in \eqref{1.1},
it can be written in terms of the magnetic field and the velocity as follows
$$
E=\nu {\rm curl}B-u \times B
$$
by the moving conductive flow in the magnetic field.
Obviously, the compressible MHD systems reduce to the full compressible
Navier-Stokes equations when there is no electro-magnetic effect(i.e. $B \equiv 0)$.

In this paper, we are concerned with the optimal decay rates for higher order
spatial derivatives of solutions to the full compressible MHD equations
in three-dimensional whole space.
Since the study of the asymptotic behavior of the MHD equations kept in step with
the Navier-Stokes equations, we recall some studies on the convergence rates for the
compressible Navier-Stokes equations with or without external forces.
When there is no external force,  the convergence rates of solutions for the
compressible Navier-Stokes equations to the steady state have been investigated extensively.
First, Matsumura and Nishida \cite{Matsumura-Nishida1} established global existence
of small solutions in $H^3$-norm and proved that the first order spatial derivatives of solutions
in $H^1$-norm converges to zero as the time goes to infinity in three-dimensional whole space.
At the same time, Matsumura and Nishida \cite{Matsumura-Nishida2} obtained the following
convergence rate for all $t \ge 0$,
$$
\|(\rho-1, u, \theta-1)(t)\|_{H^2}\lesssim (1+t)^{-\frac{3}{4}},
$$
if the small initial disturbance belongs to $H^3(\mathbb{R}^3) \cap L^1(\mathbb{R}^3)$.
For the small initial perturbation belongs to $H^3$ only,
Matsumura \cite{Matsumura} took  weighted energy method to show the time decay rates
$$
\|\nabla^k(\rho-1, u, \theta-1)(t)\|_{L^2}\lesssim (1+t)^{-\frac{k}{2}}
$$
for $k=1,2$, and
$$
\|(\rho-1,u, \theta-1)(t)\|_{L^\infty}\lesssim (1+t)^{-\frac{3}{4}}.
$$
For the same system, Ponce \cite{Ponce} gave the optimal $L^p$ convergence rate
$$
\|\nabla^l(\rho-1, u, \theta-1)(t)\|_{L^p}
\lesssim (1+t)^{-\frac{n}{2}\left(1-\frac{1}{p}\right)-\frac{l}{2}}
$$
for $2 \le p \le \infty$ and $l = 0, 1, 2$, if the small initial disturbance
belongs to $H^s(\mathbb{R}^n) \cap W^{s,1}(\mathbb{R}^n)$ with the integer
$s \ge [n/2]+ 3$ and the space dimension $n = 2$ or $3$.
In order to establish optimal decay rates for higher order spatial derivatives
of solutions, Guo and Wang \cite{Guo-Wang} developed a general energy method
to build the time convergence rates as follows
$$
\|\nabla^l(\rho-1, u)(t)\|_{H^{N-l}}^2 \lesssim (1+t)^{-(l+s)}
$$
for $0 \le l \le N-1$ by assuming the initial perturbation
are bounded in $\dot{H}^{-s}(s\in [0, \frac{3}{2}))$-norm
instead of $L^1$-norm.
On the other hand, the study of large-time behavior in $L^p(1\le p \le \infty)$ spaces
and pointwise estimates were developed in \cite{{Hoff-Zumbrun1},{Hoff-Zumbrun2},{Liu-Wang}}.
For example, Hoff and Zumbrun \cite{Hoff-Zumbrun1} studied the isentropic viscous fluid
in $\mathbb{R}^n(n \ge 2)$ and obtained
$$
\|(\rho-1, \rho u)(t)\|_{L^p}
\lesssim
\left\{
\begin{aligned}
&  t^{-\frac{n}{2}\left(1-\frac{1}{p}\right)},&2\le p\le \infty,\\
&  t^{-\frac{n}{2}\left(1-\frac{1}{p}\right)+\frac{n-1}{4}\left(\frac{p}{2}-1\right)}L_n(t),
  &1\le p <2,\\
\end{aligned}
\right.
$$
for all large $t > 0$, if the small initial disturbance belongs to
$H^s(\mathbb{R}^n) \cap L^1(\mathbb{R}^n)$ with the integer
$s\ge[n/2] + 3$, where $L_n(t)$ equals $\log(1 + t)$ if $n = 2$ and $1$ otherwise.
This result was later generalized by Kobayashi and Shibata \cite{Kobayashi-Shibata}
and Kagei and Kobayashi \cite{{Kagei-Kobayashi1},{Kagei-Kobayashi2}}
to the viscous and heat-conductive fluid and also to the half space problem but
without the smallness of $L^1$-norm of the initial disturbance.
When there is an external potential force $F =-\nabla \Phi(x)$, there are also some
results on the convergence rate for solutions to the compressible viscous
Navier-Stokes equations. For this case, when the initial perturbation is not
assumed in $L^1$, the analysis only on the Sobolev space $H^s(\mathbb{R}^3)$ yields a slower
(than optimal) decay \cite{{Deckelnick1},{Deckelnick2}}.
If the initial perturbation belongs to $L^1$ additionally, Duan et al. \cite{{Duan1}, {Duan2}}
established optimal decay rates for the solutions and its first order spatial
derivatives as follows

$$
\|\nabla^k(\rho-1, u, \theta-1)(t)\|_{H^{3-k}}
\lesssim (1+t)^{-\frac{3+2k}{4}},
$$
where $k=0,1$.

Motivated by the study of optimal decay rates for Navier-Stokes equations, the investigation
of time convergence rates of solutions to  the MHD equations has aroused many researchers' interests.
First of all, under the $H^3$-framework,  Li and Yu \cite{Li-Yu} and Chen and Tan \cite{Chen-Tan}
not only established the global existence of classical solutions, but also obtained
the time decay rates for the three-dimensional compressible MHD equations by assuming the initial data belong to $L^1$ and $L^q( q \in \left[1, \frac{6}{5}\right))$ respectively.
More precisely, Chen and Tan \cite{Chen-Tan} built the time decay rates
\begin{equation}\label{Decay-Chen-Tan}
\|\nabla^k (\rho-1, u, B)(t)\|_{H^{3-k}}\lesssim (1+t)^{-\frac{3}{2}\left(\frac{1}{q}-\frac{1}{2}\right)-\frac{k}{2}},
\end{equation}
where $k=0,1$. These decay rates \eqref{Decay-Chen-Tan} have also been established by
Li and Yu \cite{Li-Yu} for the case $q=1$.
Motivated by the work of Guo and Wang \cite{Guo-Wang},
Tan and Wang \cite{Tan-Wang} established
the optimal time decay rates for the higher order spatial derivatives of solutions
if the initial perturbation belongs to $H^N\cap\dot{H}^{-s}\left(N\ge3, s \in \left[0, \frac{3}{2}\right)\right)$.
More precisely, they built the following time decay rates
$$
\|\nabla^k (\rho-1, u, B)(t)\|_{H^{N-k}}\lesssim (1+t)^{-\frac{k+s}{2}},
$$
where $k=0,1,...,N-1$.
Following the spirit of work \cite{Gao-Tao-Yao}, we (see \cite{Gao-Chen-Yao})
establish the following time decay rates
for all $t \ge T^*(T^*$ is a positive constant$)$
\begin{equation}\label{Decay-Gao-Chen-Yao}
\begin{aligned}
&\|\nabla^2(\rho-1)(t)\|_{H^{1}}+\|\nabla^2 u(t)\|_{H^{1}}\lesssim (1+t)^{-\frac{7}{4}},\\
&\|\nabla^m B(t)\|_{H^{3-m}}\lesssim (1+t)^{-\frac{3+2m}{4}},\\
\end{aligned}
\end{equation}
where $m=2,3$.  It is easy to see that the time decay rates \eqref{Decay-Gao-Chen-Yao}
is better than decay rates \eqref{Decay-Chen-Tan} since \eqref{Decay-Gao-Chen-Yao}
provides faster time decay rates for the higher order spatial derivatives of solutions.
For the full compressible MHD equations \eqref{1.1}, Pu and Guo \cite{Pu-Guo}
established time decay rates for the classical solutions in three-dimensional whole space
as follows
\begin{equation}\label{Decay-Pu-Guo}
\|\nabla^k(\rho-1,u,\theta-1, B)(t)\|_{H^{3-k}}\lesssim (1+t)^{-\frac{3}{2}\left(\frac{1}{q}-\frac{1}{2}\right)-\frac{k}{2}},
\end{equation}
where $k=0,1$, if the initial perturbation are small in $H^3$-norm
and bounded in $L^q\left(q\in \left[1, \frac{6}{5}\right)\right)$-norm.

In this paper, we are concerned with optimal decay rates for the higher order
spatial derivatives of classical solutions in $L^2$-norm to the full compressible
MHD equations in three dimensional whole space.
For the classical incompressible Navier-Stokes equations, Schonbek and Wiegner
\cite{{Schonbek1},{Schonbek-Wiegner}} applied the inductive argument and Fourier
splitting method (see \cite{Schonbek2}) to establish optimal decay rates
for higher order derivatives norm after having the optimal decay rates
of solutions and its first order spatial derivatives at hand.
Motivated by \cite{{Schonbek1},{Schonbek-Wiegner}},  we move the nonlinear terms
to the right hand side of \eqref{1.1}$_4$ and deal with the nonlinear terms
as external force with the property on fast time decay rates.
Then, the application of Fourier splitting method helps us to establish optimal
time decay rates for the higher order spatial derivatives of magnetic field in $L^2$-norm.
Since the equation \eqref{1.1} is hyperbolic-parabolic type,
then the optimal decay rate for the second order spatial derivatives of solutions
are somewhat complicated. More precisely, denoting $\varrho=\rho-1$ and $\sigma=\theta-1$,
the \eqref{1.1} transforms into the system \eqref{eq1}.
Then, for the homogeneous system \eqref{eq1}(i.e. $S_1=S_2=S_3=S_4=0$),
it is easy to establish following energy inequality
\begin{equation}\label{1.8}
\frac{d}{dt}\|\nabla^k (\varrho, u, \sigma)\|_{L^2}^2
+C\|\nabla^{k+1} (u, \sigma)\|_{L^2}^2 \le 0,
\end{equation}
where $k=2, 3$.
In order to apply the Fourier splitting method to build optimal decay rate for the
second order derivatives norm of solution, we need to rediscover the dissipative
estimates for $\varrho$. Then, it is easy to verify the following
differential inequality
\begin{equation}\label{1.9}
\frac{d}{dt}\int \nabla^2 u \cdot \nabla^3 \varrho dx
+C\|\nabla^3 \varrho\|_{L^2}^2\le \|\nabla^3 u\|_{H^1}^2+\|\nabla^3 \sigma\|_{L^2}^2.
\end{equation}
The combination of \eqref{1.8} and \eqref{1.9} yields directly
\begin{equation}\label{1.10}
\frac{d}{dt} \mathcal{E}_2^3(t)
+C(\|\nabla^3 \varrho\|_{L^2}^2+ \|\nabla^3 (u, \sigma)\|_{H^1}^2)
\le 0,
\end{equation}
where $\mathcal{E}_2^3(t)$ is equivalent to $\|\nabla^2(\varrho, u, \sigma)\|_{H^1}^2$.
Following the spirit of the Fourier splitting method, we need to transform the inequality
\eqref{1.10} as the form
$$
\frac{d}{dt} \mathcal{E}_2^3(t)
+\frac{C}{2}(2\|\nabla^3 \varrho\|_{L^2}^2+ \|\nabla^3 (u, \sigma)\|_{H^1}^2)
\le 0.
$$
Then, it is easy to establish optimal
decay rates for the second order spatial derivatives of solutions
in $L^2$-norm(see Lemma \ref{lemma2.6} for detail).
Finally, one also establishes time convergence rates for the mixed space-time
derivatives of solutions.

\textbf{Notation:} In this paper, we use $H^s(\mathbb{R}^3)( s\in \mathbb{R})$ to
denote the usual Sobolev spaces
with norm $\|\cdot\|_{H^s}$ and $L^p(\mathbb{R}^3)(1\le p \le \infty)$ to denote the usual $L^p$ spaces with norm $\| \cdot \|_{L^p}$. We define
$$
\nabla^k v=
\left\{\left.\partial_x^\alpha v_i\right||\alpha|=k,~i=1,2,3\right\}
,~v=(v_1, v_2, v_3).
$$
One also denotes the Fourier transform by $\mathscr{F}(f):=\hat{f}$ .
The notation $a \lesssim b$ means that $a \le C b$ for a universal constant $C>0$ independent of
time $t$.
The notation $a \approx b$ means $a \lesssim b$ and $b \lesssim a$.
For the sake of simplicity, we write $\int f dx:=\int _{\mathbb{R}^3} f dx.$

The main results on optimal decay rates for higher order spatial derivatives of solutions
to the full compressible MHD equations \eqref{1.1}-\eqref{1.3} can be stated as follows.

\begin{theo}\label{Theorem1.1}
Assume that the initial data $(\rho_0-1,u_0,\theta_0-1, B_0)\in H^3
\cap{L^q} (q \in \left[1, \frac{6}{5}\right))$ and there exists a
small constant $\delta_0>0$ such that
\begin{equation}\label{smallness-condition}
\|(\rho_0-1,u_0,\theta_0-1, B_0)\|_{H^3} \le \delta_0,
\end{equation}
then the global classical solution $(\rho, u, \theta, B)$ of
Cauchy problem \eqref{1.1}-\eqref{1.3} has the following decay rates
for all $t \ge T^*(T^*~\text{is~a~positive~constant})$,
\begin{equation}\label{Decay1}
\begin{aligned}
&\|\nabla^2 (\rho-1, u, \theta-1)(t)\|_{H^1} \le  C(1+t)^{-\frac{3}{2}\left(\frac{1}{q}-\frac{1}{2}\right)-1},\\
&\|\nabla^k  B (t)\|_{L^2} \le  C(1+t)^{-\frac{3}{2}\left(\frac{1}{q}-\frac{1}{2}\right)-\frac{k}{2}},
\end{aligned}
\end{equation}
where $k=2,3$.
\end{theo}

\begin{rema}
Obviously, \eqref{Decay1} provides faster time decay rates for the higher order spatial
derivatives of classical solutions than \eqref{Decay-Pu-Guo}.
Hence, the results in Theorem \ref{Theorem1.1} improve the work of Pu and Guo \cite{Pu-Guo}.
\end{rema}

\begin{rema}
By virtue of the Sobolev inequality and the results in Theorem \ref{Theorem1.1},
the global solution $(\rho, u, \theta, B)$ of problem \eqref{1.1}-\eqref{1.3} has the time decay rates
\begin{equation*}
\begin{aligned}
&\|(\rho-1, u, \theta-1)(t)\|_{L^p}
\le C(1+t)^{-\frac{3}{2}\left(\frac{1}{q}-\frac{1}{p}\right)},\\
&\|\nabla^k B(t)\|_{L^p}
\le C(1+t)^{{-\frac{3}{2}\left(\frac{1}{q}-\frac{1}{p}\right)}-\frac{k}{2}},
\end{aligned}
\end{equation*}
where $2 \le p \le \infty,$ and $\ k=0,1.$
\end{rema}

\begin{rema}
Although we only established the time decay rates under the $H^3$-framework
in Theorem \ref{Theorem1.1}, the method here can be
applied to the $H^N(N \ge 3)$-framework just following the idea
as Gao et al. \cite{Gao-Tao-Yao}. Hence,
if $(\rho_0-1, u_0, \theta_0-1, B_0)\in H^N \cap L^q(N \ge3, q \in \left[1, \frac{6}{5}\right))$
and there exists a small constant $\varepsilon_0>0$
such that
\begin{equation*}
\|(\rho_0-1,u_0,\theta_0-1, B_0)\|_{H^N} \le \varepsilon_0,
\end{equation*}
then the global classical solutions have the time decay rates
$$
\begin{aligned}
&\|\nabla^k(\rho-1,  u, \theta-1)(t)\|_{H^{N-k}}
\le C(1+t)^{-\frac{3}{2}\left(\frac{1}{q}-\frac{1}{2}\right)-\frac{k}{2}},\\
&\|\nabla^m B(t)\|_{H^{N-m}}
\le C(1+t)^{-\frac{3}{2}\left(\frac{1}{q}-\frac{1}{2}\right)-\frac{m}{2}},\\
\end{aligned}
$$
where $k=0, 1,..., N-1,$ and $m=0,1,2,..., N.$
\end{rema}

Finally, we establish decay rates for the mixed space-time derivatives of solutions
to the Cauchy problem \eqref{1.1}-\eqref{1.3}.

\begin{theo}\label{Theorem1.2}
Under all the assumptions in Theorem \ref{Theorem1.1}, then the global classical solution
$(\rho,u,\theta, B)$ of Cauchy problem \eqref{1.1}-\eqref{1.3} has the time decay rates
\begin{equation}\label{Decay3}
\begin{aligned}
&\|\nabla^k \rho_t (t)\|_{H^{2-k}}
+\|\nabla^k u_t(t)\|_{L^2}
+\|\nabla^k \theta_t(t)\|_{L^2}
\le C(1+t)^{-\frac{3}{2}\left(\frac{1}{q}-\frac{1}{2}\right)-\frac{k+1}{2}},\\
&\|\nabla^k B_t(t)\|_{L^2} \le C(1+t)^{-\frac{3}{2}\left(\frac{1}{q}-\frac{1}{2}\right)-\frac{k+2}{2}},\\
\end{aligned}
\end{equation}
where $k=0,1$.
\end{theo}

\begin{rema}
By virtue of the Sobolev inequality and the results in Theorem \ref{Theorem1.2},
the global solution $(\rho, u, \theta, B)$ of problem \eqref{1.1}-\eqref{1.3} has the time decay rates
$$
\begin{aligned}
&\|(\rho_t, u_t, \theta_t)(t)\|_{L^p}
\le C(1+t)^{-\frac{3}{2}\left(\frac{1}{q}-\frac{1}{p}\right)-\frac{1}{2}},\\
&\|B_t(t)\|_{L^p} \le C(1+t)^{{-\frac{3}{2}\left(\frac{1}{q}-\frac{1}{p}\right)}-1},
\end{aligned}
$$
where $2 \le p \le 6$.
\end{rema}

The rest of this paper is organized as follows. In section \ref{section2},
we establish the optimal time decay rates
for the higher order spatial derivatives of global classical solutions.
In section \ref{section3}, one hopes to build the
time convergence rates for the mixed space-time derivatives of solutions.

\section{Proof of Theorem \ref{Theorem1.1}}\label{section2}

\quad In this section, we will establish the optimal time decay rates for the higher
order spatial derivatives of solutions.
By computing directly, it is easy to deduce
$$
({\rm curl} B) \times B=(B\cdot \nabla)B-\frac{1}{2}\nabla (|B|^2),
$$
and
$$
{\rm curl}{(u \times B)}=u({\rm div}B)-(u\cdot \nabla) B+(B\cdot \nabla)u-B({\rm div}u).
$$
Denoting $\varrho=\rho-1$ and $\sigma=\theta-1$, the original full MHD equations \eqref{1.1}
can be rewritten in the perturbation form as follows
\begin{equation}\label{eq1}
\left\{
\begin{aligned}
&\varrho_t+{\rm div}u=S_1,\\
&u_t-\mu \Delta u-(\mu+\lambda)\nabla {\rm div}u+\nabla \varrho+\nabla \sigma=S_2,\\
&\sigma_t-\kappa \Delta \sigma+{\rm div}u=S_3,\\
&B_t-\nu \Delta B=S_4, \ {\rm div}B=0.
\end{aligned}
\right.
\end{equation}
Here $S_i(i=1,2,3,4)$ are defined as
\begin{equation}\label{eq2}
\left\{
\begin{aligned}
S_1=&-\varrho {\rm div}u-u\cdot \nabla \varrho,\\
S_2=&-u\cdot \nabla u-h(\varrho)[\mu\Delta u+(\mu+\lambda)\nabla {\rm div}u]
      -E(\varrho, \sigma)\nabla \varrho-F(\varrho, \sigma)\nabla \sigma\\
    & +g(\varrho)\left[B\cdot \nabla B-\frac{1}{2}\nabla(|B|^2)\right],\\
S_3=&-u\cdot \nabla \sigma- h(\varrho)(\kappa \Delta \sigma)-G(\varrho, \sigma){\rm div}u
      +g(\varrho)\left[2\mu|D(u)|^2+\lambda({\rm div}u)^2\right],\\
    &+g(\varrho)(\nu |{\rm curl} B|^2),\\
S_4=&-u \cdot \nabla B+B\cdot \nabla u -B {\rm div}u,
\end{aligned}
\right.
\end{equation}
where the functions of $\varrho$ and $\sigma$ are defined as
\begin{equation}\label{eq3}
\begin{aligned}
&h(\varrho)=\frac{\varrho}{\varrho+1},\quad
\ g(\varrho)=\frac{1}{\varrho+1},\quad
\ G(\varrho, \sigma)=\frac{(\sigma+1)P_\sigma(\varrho+1,\sigma+1)}{\varrho+1}-1,\\
&\ E(\varrho, \sigma)=\frac{P_\varrho(\varrho+1,\sigma+1)}{\varrho+1}-1,\quad
\ F(\varrho, \sigma)=\frac{P_\sigma(\varrho+1,\sigma+1)}{\varrho+1}-1.
\end{aligned}
\end{equation}
To complete the system \eqref{eq1}, the initial data are given by
\begin{equation}\label{eq4}
\left.(\varrho, u, \sigma, B)(x,t)\right|_{t=0}=(\varrho_0, u_0, \sigma_0, B_0)(x)
\rightarrow(0,0,0,0) \quad {\text as } \quad |x|\rightarrow \infty.
\end{equation}

First of all, Pu and Guo(see \cite{Pu-Guo} on Page $521$ ) have established the following estimates
\begin{equation}\label{smallness-condition1}
\|(\varrho, u, \sigma, B)\|_{H^3}\le C \|(\varrho_0, u_0, \sigma_0, B_0)\|_{H^3},
\end{equation}
which, together with the smallness of $\delta_0$(see \eqref{smallness-condition})
and Sobolev inequality, yields directly
$$
\frac{1}{2}\le \varrho+1 \le \frac{3}{2}.
$$
Hence, we immediately have
\begin{equation}\label{2.6}
\quad |g(\varrho)| \le C, \ |h(\varrho)| \le C|\varrho| \ \text{and }\
|G(\varrho, \sigma)|, |E(\varrho, \sigma)|, |F(\varrho, \sigma)|\le C(|\varrho|+|\sigma|),
\end{equation}
which will be used frequently to derive the temporal decay rates.
Furthermore, it is also easy to deduce
$$
\begin{aligned}
|\nabla E(\varrho, \sigma)|
&=|E_\varrho(\varrho, \sigma)\nabla \varrho+E_\sigma(\varrho, \sigma)\nabla \sigma|
\lesssim |\nabla \varrho|+|\nabla \sigma|,\\
|\nabla^2 E(\varrho, \sigma)|
&=|\nabla(E_\varrho(\varrho, \sigma)\nabla \varrho+E_\sigma(\varrho, \sigma)\nabla \sigma)|\\
&=|(E_{\varrho \varrho}(\varrho, \sigma)\nabla \varrho+E_{\varrho \sigma}(\varrho, \sigma)\nabla \sigma)\nabla \varrho
   +E_\varrho(\varrho, \sigma)\nabla^2 \varrho\\
&\quad \quad  +(E_{\sigma \varrho}(\varrho, \sigma)\nabla \varrho+E_{\sigma\sigma}(\varrho, \sigma)\nabla \sigma)\nabla \sigma
+E_\sigma(\varrho, \sigma)\nabla^2 \sigma |\\
&\lesssim |\nabla \varrho|^2+|\nabla \sigma|^2+|\nabla^2 \varrho|+|\nabla^2 \sigma|.
\end{aligned}
$$
In the same manner, we also obtain
\begin{equation}\label{2.7}
\begin{aligned}
&|\nabla h(\varrho)|,|\nabla g(\varrho)|\lesssim |\nabla \varrho|,\\
&|\nabla^2 h(\varrho)|,|\nabla^2 g(\varrho)|\lesssim |\nabla \varrho|^2+|\nabla^2 \varrho|,\\
&|\nabla E(\varrho, \sigma)|, |\nabla F(\varrho, \sigma)|, |\nabla G(\varrho, \sigma)|
 \lesssim |\nabla \varrho|+|\nabla \sigma|,\\
&|\nabla^2 E(\varrho, \sigma)|,|\nabla^2 F(\varrho, \sigma)|,|\nabla^2 G(\varrho, \sigma)|
\lesssim |\nabla \varrho|^2+|\nabla \sigma|^2+|\nabla^2 \varrho|+|\nabla^2 \sigma|.
\end{aligned}
\end{equation}

We state the classical Sobolev interpolation of the Gagliardo-Nirenberg inequality, refer to \cite{Nirenberg}.
\begin{lemm}\label{lemma2.1}
Let $0 \le m, \alpha \le l$ and the function $f\in C_0^\infty(\mathbb{R}^3)$, then we have
\begin{equation}\label{GN}
\|\nabla^\alpha f\|_{L^p} \lesssim \|\nabla^m f \|_{L^2}^{1-\theta} \|\nabla^l f \|_{L^2}^\theta,
\end{equation}
where $0\le \theta \le 1$ and $\alpha$ satisfy
\begin{equation*}
\frac{1}{p}-\frac{\alpha}{3}=\left(\frac{1}{2}-\frac{m}{3}\right)(1-\theta)
+\left(\frac{1}{2}-\frac{l}{3}\right)\theta.
\end{equation*}
\end{lemm}

First of all, we establish the optimal time decay rates for the second order spatial derivatives of magnetic field.

\begin{lemm}\label{lemma2.2}
Under the assumptions of Theorem \ref{Theorem1.1}, the magnetic field has the time decay rates
for all $t \ge 0$,
\begin{equation}\label{231}
\|\nabla^2 B (t)\|_{H^1}^2 \le C (1+t)^{-3\left(\frac{1}{q}-\frac{1}{2}\right)-2}.
\end{equation}
\end{lemm}
\begin{proof}
Taking second order spatial derivatives to $\eqref{eq1}_4$,
multiplying the resultant identity by $\nabla^k B$
and integrating over $\mathbb{R}^3$, then we have
\begin{equation}\label{222}
\begin{aligned}
&\frac{1}{2}\frac{d}{dt}\int |\nabla^2 B|^2 dx+\nu \int |\nabla^3 B|^2 dx\\
&=-\int \nabla^2 (u\cdot \nabla B)\nabla^2 B dx
  +\int \nabla^2 (B\cdot \nabla u)\nabla^2 B dx\\
&\quad -\int \nabla^2 (B{\rm div}u)\nabla^2 B dx\\
&=I_1+I_2+I_3.
\end{aligned}
\end{equation}
The application of decay rate \eqref{Decay-Pu-Guo}, Holder, Sobolev and Cauchy inequalities yields immediately
\begin{equation}\label{223}
\begin{aligned}
I_1
&=\int (\nabla u \nabla B+u \nabla^2 B)\nabla^3 B dx\\
&\lesssim (\|\nabla u\|_{L^3}\|\nabla B\|_{L^6}+\|u\|_{L^6}\|\nabla^2 B\|_{L^3})\|\nabla^3 B\|_{L^2}\\
&\lesssim \|\nabla u\|_{L^3}^2\|\nabla^2 B\|_{L^2}^2
           +\|\nabla u\|_{L^2}^2\|\nabla^2 B\|_{H^1}^2
           +\varepsilon\|\nabla^3 B\|_{L^2}^2\\
&\lesssim \|\nabla u\|_{H^1}^2\|\nabla^2 B\|_{H^1}^2+\varepsilon\|\nabla^3 B\|_{L^2}^2\\
&\lesssim (1+t)^{-\frac{3}{q}+\frac{1}{2}}(1+t)^{-\frac{3}{q}+\frac{1}{2}}
          +\varepsilon\|\nabla^3 B\|_{L^2}^2\\
&\lesssim (1+t)^{-\frac{6-q}{q}}+\varepsilon\|\nabla^3 B\|_{L^2}^2.
\end{aligned}
\end{equation}
In the same manner, it is easy to deduce
\begin{equation}\label{224}
\begin{aligned}
I_2
&=-\int (\nabla B \nabla u+B \nabla^2 u)\nabla^3 B dx\\
&\lesssim (\|\nabla u\|_{L^3}\|\nabla B\|_{L^6}+\|B\|_{L^6}\|\nabla^2 u\|_{L^3})\|\nabla^3 B\|_{L^2}\\
&\lesssim \|\nabla u\|_{H^1}^2\|\nabla^2 B\|_{L^2}^2
          +\|\nabla B\|_{L^2}^2\|\nabla^2 u\|_{H^1}^2+\varepsilon\|\nabla^3 B\|_{L^2}^2\\
&\lesssim \|\nabla u\|_{H^2}^2\|\nabla B\|_{H^1}^2+\varepsilon\|\nabla^3 B\|_{L^2}^2\\
&\lesssim (1+t)^{-\frac{3}{q}+\frac{1}{2}}(1+t)^{-\frac{3}{q}+\frac{1}{2}}
          +\varepsilon\|\nabla^3 B\|_{L^2}^2\\
&\lesssim (1+t)^{-\frac{6-q}{q}}+\varepsilon\|\nabla^3 B\|_{L^2}^2,
\end{aligned}
\end{equation}
and
\begin{equation}\label{225}
\begin{aligned}
I_3
&\lesssim \|\nabla u\|_{H^2}^2\|\nabla B\|_{H^1}^2+\varepsilon\|\nabla^3 B\|_{L^2}^2\\
&\lesssim (1+t)^{-\frac{6-q}{q}}+\varepsilon\|\nabla^3 B\|_{L^2}^2.
\end{aligned}
\end{equation}
Substituting \eqref{223}-\eqref{225} into \eqref{222} and
choosing $\varepsilon$ small enough, we get
\begin{equation}\label{226}
\frac{d}{dt}\int |\nabla^2 B|^2 dx+\nu \int |\nabla^3 B|^2 dx
\lesssim (1+t)^{-\frac{6-q}{q}}.
\end{equation}
Taking third order spatial derivatives to $\eqref{eq1}_4$,
multiplying the resulting identity by $\nabla^3 B$
and integrating over $\mathbb{R}^3$, then we have
\begin{equation}\label{227}
\begin{aligned}
&\frac{1}{2}\frac{d}{dt}\int |\nabla^3 B|^2 dx+\nu \int |\nabla^4 B|^2 dx\\
&=-\int \nabla^3 (u\cdot \nabla B)\nabla^3 B dx
  +\int \nabla^3 (B\cdot \nabla u)\nabla^3 B dx\\
&\quad -\int \nabla^3 (B{\rm div}u)\nabla^3 B dx\\
&=I\!I_1+I\!I_2+I\!I_3.
\end{aligned}
\end{equation}
By virtue of \eqref{smallness-condition1}, Holder, Sobolev and Cauchy inequalities, it arrives at
\begin{equation}\label{228}
\begin{aligned}
I\!I_1
&=\int (\nabla^2 u \nabla B+2\nabla u \nabla^2 B+u\nabla^3 B)\nabla^4 B dx\\
&\lesssim (\|\nabla^2 u\|_{L^3}\|\nabla B\|_{L^6}+\|\nabla u\|_{L^3}\|\nabla^2 B\|_{L^6}
         +\|u\|_{L^3}\|\nabla^3 B\|_{L^6})\|\nabla^4 B\|_{L^2}\\
&\lesssim \|\nabla^2 u\|_{H^1}^2\|\nabla^2 B\|_{L^2}^2+\|\nabla u\|_{H^1}^2\|\nabla^3 B\|_{L^2}^2
+(\varepsilon+\delta)\|\nabla^4 B\|_{L^2}^2.
\end{aligned}
\end{equation}
Similarly, it is easy to deduce
\begin{equation}\label{229}
\begin{aligned}
I\!I_2
&=-\int (\nabla^2 B\nabla u+2\nabla B \nabla^2 u+B\nabla^3 u)\nabla^4 B dx\\
&\lesssim (\|\nabla^2 B\|_{L^6}\|\nabla u\|_{L^3}+\|\nabla B\|_{L^6}\|\nabla^2 u\|_{L^3}
           +\|B\|_{L^\infty}\|\nabla^3 u\|_{L^2})\|\nabla^4 B\|_{L^2}\\
&\lesssim (\|\nabla^3 B\|_{L^2}\|\nabla u\|_{L^3}+\|\nabla^2 B\|_{L^2}\|\nabla^2 u\|_{L^3}
           +\|\nabla B\|_{H^1}\|\nabla^3 u\|_{L^2})\|\nabla^4 B\|_{L^2}\\
&\lesssim \|\nabla^2 B\|_{H^1}^2\|\nabla u\|_{H^2}^2
           +\|\nabla B\|_{H^1}^2\|\nabla^3 u\|_{L^2}^2
           +\varepsilon\|\nabla^4 B\|_{L^2}^2,
\end{aligned}
\end{equation}
and
\begin{equation}\label{2210}
I\!I_3\lesssim
\|\nabla^2 B\|_{H^1}^2\|\nabla u\|_{H^2}^2
           +\|\nabla B\|_{H^1}^2\|\nabla^3 u\|_{L^2}^2
           +\varepsilon\|\nabla^4 B\|_{L^2}^2.
\end{equation}
Substituting \eqref{228}-\eqref{2210} into \eqref{227}, then we obtain
\begin{equation}\label{third-order}
\frac{d}{dt}\int |\nabla^3 B|^2 dx+\nu\int |\nabla^4 B|^2 dx
\lesssim \|\nabla^2 B\|_{H^1}^2\|\nabla u\|_{H^2}^2
           +\|\nabla B\|_{H^1}^2\|\nabla^3 u\|_{L^2}^2,
\end{equation}
which, together with the decay rates \eqref{Decay-Pu-Guo}, yields directly
\begin{equation}\label{2211}
\frac{d}{dt}\int |\nabla^3 B|^2 dx+\nu\int |\nabla^4 B|^2 dx
\lesssim (1+t)^{-\frac{6-q}{q}}.
\end{equation}
Adding \eqref{227} to \eqref{2211}, it arrives at
\begin{equation}\label{2212}
\frac{d}{dt}\int (|\nabla^2 B|^2+|\nabla^3 B|^2) dx+\nu \int (|\nabla^3 B|^2+|\nabla^4 B|^2) dx
\lesssim (1+t)^{-\left(\frac{6}{q}-1\right)}.
\end{equation}
For some constant $R$ defined below, denoting the time sphere $S_0$
(see Schonbek \cite{Schonbek2}) by
$$
S_0:=\left\{\left.\xi \in \mathbb{R}^3\right| |\xi| \le \left(\frac{R}{1+t}\right)^{\frac{1}{2}}\right\},
$$
then we have
$$
\begin{aligned}
\int_{\mathbb{R}^3} |\nabla^3 B|^2 dx
&\ge \int_{\mathbb{R}^3/S_0}|\xi|^6 |\hat{B}|^2 d\xi \\
&\ge \frac{R}{1+t}\int_{\mathbb{R}^3} |\xi|^4 |\hat{B}|^2 d\xi
     -\left(\frac{R}{1+t}\right)^2\int_{S_0} |\xi|^2 |\hat{B}|^2d\xi\\
&\ge \frac{R}{1+t}\int_{\mathbb{R}^3} |\nabla^2 B|^2  dx
     -\left(\frac{R}{1+t}\right)^2\int_{\mathbb{R}^3} |\nabla B|^2dx,\\
\end{aligned}
$$
or equivalently
\begin{equation}\label{FSM}
\int |\nabla^3 B|^2 dx \ge \frac{R}{1+t}\int |\nabla^2 B|^2  dx-\left(\frac{R}{1+t}\right)^2\int |\nabla B|^2dx.
\end{equation}
Similarly, it is easy to deduce
\begin{equation}\label{2213}
\int |\nabla^4 B|^2 dx\ge \frac{R}{1+t}\int |\nabla^3 B|^2 dx-\left(\frac{R}{1+t}\right)^2\int |\nabla^2 B|^2dx.
\end{equation}
Substituting \eqref{FSM} and \eqref{2213} into \eqref{2212}
and applying the time decay rates \eqref{Decay-Pu-Guo}, we have
$$
\begin{aligned}
&\frac{d}{dt}\int (|\nabla^2 B|^2+|\nabla^3 B|^2) dx+\frac{R\nu}{1+t}\int (|\nabla^2 B|^2+|\nabla^3 B|^2) dx\\
&\lesssim \frac{R^2 \nu}{(1+t)^2}\int (|\nabla B|^2+|\nabla^2 B|^2) dx
           +(1+t)^{-\left(\frac{6}{q}-1\right)}\\
&\lesssim (1+t)^{-2}(1+t)^{-\frac{3}{q}+\frac{1}{2}}
           +(1+t)^{-\left(\frac{6}{q}-1\right)}\\
&\lesssim (1+t)^{-\left(\frac{3}{q}+\frac{3}{2}\right)}+(1+t)^{-\left(\frac{6}{q}-1\right)}.
\end{aligned}
$$
By virtue of $q\in \left[1, \frac{6}{5}\right)$, then it is easy to see that
\begin{equation}\label{Large-Small-Relation}
\left(\frac{6}{q}-1\right)-\left(\frac{3}{q}+\frac{3}{2}\right)
=\frac{5\left(\frac{6}{5}-q\right)}{2q} \ge 0.
\end{equation}
Thus, we have the following estimates
\begin{equation}\label{2214}
\frac{d}{dt}\int (|\nabla^2 B|^2+|\nabla^3 B|^2) dx
+\frac{R\nu}{1+t}\int (|\nabla^2 B|^2+|\nabla^3 B|^2) dx
\lesssim (1+t)^{-\left(\frac{3}{q}+\frac{3}{2}\right)}.
\end{equation}
If choosing
\begin{equation}\label{2215}
R=\frac{3+q}{q\nu}
\end{equation}
in \eqref{2214}, then we get
\begin{equation}\label{2216}
\frac{d}{dt}\int (|\nabla^2 B|^2+|\nabla^3 B|^2) dx+\frac{3+q}{q(1+t)}\int (|\nabla^2 B|^2+|\nabla^3 B|^2) dx
\lesssim (1+t)^{-\left(\frac{3}{q}+\frac{3}{2}\right)}.
\end{equation}
Multiplying \eqref{2216} by $(1+t)^{\frac{3+q}{q}}$, it arrives at
$$
\frac{d}{dt}\left[(1+t)^{\frac{3+q}{q}}\|\nabla^2 B\|_{H^1}^2\right]\le C(1+t)^{-\frac{1}{2}},
$$
which, integrating over $[0, t]$, gives
$$
\|\nabla^2 B (t)\|_{H^1}^2 \le
(1+t)^{-\left(\frac{3}{q}+1\right)}\left[\|\nabla^2 B_0\|_{H^1}^2+C(1+t)^{\frac{1}{2}}\right],
$$
or equivalently
$$
\|\nabla^2 B (t)\|_{H^1}^2 \le C(1+t)^{-\left(\frac{3}{q}+\frac{1}{2}\right)}.
$$
Therefore, we complete the proof of the lemma.
\end{proof}

Next, we establish the following differential inequality for the second order
spatial derivatives of solutions.

\begin{lemm}\label{lemma2.3}
Under all the assumptions in Theorem \ref{Theorem1.1}, then we have
\begin{equation}\label{231}
\frac{d}{dt}\|\nabla^2(\varrho, u, \sigma)\|_{L^2}^2
+(\mu \|\nabla^3 u\|_{L^2}^2+\kappa \|\nabla^3 \sigma\|_{L^2}^2)
\lesssim \varepsilon\|\nabla^3 \varrho\|_{L^2}^2+(1+t)^{-\frac{6-q}{q}}.
\end{equation}
\end{lemm}
\begin{proof}
Taking $k-$th$(k=2,3)$ spatial derivatives on both hand sides of
\eqref{eq1}$_1$,\eqref{eq2}$_2$ and \eqref{eq2}$_3$,
multiplying the resultant identity by $\nabla^k \varrho$, $\nabla^k u$
and $\nabla^k \sigma$ respectively and integrating
over $\mathbb{R}^3$, then we have
\begin{equation}\label{232}
\begin{aligned}
&\frac{1}{2}\frac{d}{dt}\!\int\! (|\nabla^k \varrho|^2+|\nabla^k u|^2+|\nabla^k \sigma|^2) dx
+\!\int\! (\mu |\nabla^{k+1}u|^2+(\mu+\lambda)|\nabla^k {\rm div}u|^2+\kappa |\nabla^{k+1} \sigma|^2) dx\\
&=\int \nabla^k S_1 \cdot \nabla^k \varrho dx
+\int \nabla^k S_2  \cdot \nabla^k u dx
+\int \nabla^k S_3  \cdot \nabla^k \sigma dx.
\end{aligned}
\end{equation}
Taking $k=2$ in \eqref{232}, it arrives at
\begin{equation}\label{233}
\begin{aligned}
&\frac{1}{2}\frac{d}{dt}\int (|\nabla^2 \varrho|^2+|\nabla^2 u|^2+|\nabla^2 \sigma|^2) dx
+\int (\mu |\nabla^3u|^2+(\mu+\lambda)|\nabla^2 {\rm div}u|^2+\kappa |\nabla^3 \sigma|^2) dx\\
&=\!\int\! \nabla^2 (-\varrho {\rm div}u-u\cdot \nabla \varrho) \nabla^2 \varrho dx
  +\!\int\! \nabla^2 (-u\cdot \nabla u-
                   h(\varrho)[\mu\Delta u+(\mu+\lambda)\nabla {\rm div}u]) \nabla^2 u dx\\
&\quad +\int \nabla^2 (-E(\varrho, \sigma)\nabla \varrho-F(\varrho, \sigma)\nabla \sigma
             +g(\varrho)[B\cdot \nabla B-\frac{1}{2}\nabla(|B|^2)]) \nabla^2 u dx\\
&\quad +\int \nabla^2 (-u\cdot \nabla \sigma- h(\varrho)(\kappa \Delta \sigma)
       -G(\varrho, \sigma){\rm div}u
      +g(\varrho)\left[2\mu|D(u)|^2+\lambda({\rm div}u)^2\right]) \nabla^2 \sigma dx\\
&\quad +\int\nabla^2 (g(\varrho)(\nu |{\rm curl} B|^2)) \nabla^2 \sigma dx.
\end{aligned}
\end{equation}
Integrating by part and applying \eqref{Decay-Pu-Guo},
Holder, Sobolev  and Cauchy inequalities, we obtain
\begin{equation}\label{234}
\begin{aligned}
&\int \nabla^2 (-\varrho {\rm div}u-u\cdot \nabla \varrho) \nabla^2 \varrho dx\\
&=-\int(\varrho \nabla^2 u+2\nabla \varrho \nabla u +u \nabla^2 \varrho)\nabla^3 \varrho dx\\
&\lesssim (\|\varrho\|_{L^6}\|\nabla^2 u\|_{L^3}+\|\nabla \varrho\|_{L^6}\|\nabla u\|_{L^3}
           +\|u\|_{L^6}\|\nabla^2 \varrho\|_{L^3})\|\nabla^3 \varrho\|_{L^2}\\
&\lesssim \|\nabla \varrho\|_{L^2}^2\|\nabla^2 u\|_{H^1}^2
           +\|\nabla^2 \varrho\|_{L^2}^2\|\nabla u\|_{H^1}^2
          +\|\nabla u\|_{L^2}^2\|\nabla^2 \varrho\|_{H^1}^2+\varepsilon\|\nabla^3 \varrho\|_{L^2}^2\\
&\lesssim (1+t)^{-\frac{3}{q}+\frac{1}{2}}(1+t)^{-\frac{3}{q}+\frac{1}{2}}
           +\varepsilon\|\nabla^3 \varrho\|_{L^2}^2\\
&\lesssim (1+t)^{-\frac{6-q}{q}}+\varepsilon\|\nabla^3 \varrho\|_{L^2}^2.
\end{aligned}
\end{equation}
Following the idea as \eqref{234}, it is easy to deduce
\begin{equation}\label{235}
\begin{aligned}
\int \nabla^2 (-u \cdot \nabla u)\nabla^2 udx
\lesssim \|\nabla u\|_{H^1}^2\|\nabla^2 u\|_{H^1}^2+\varepsilon\|\nabla^3 u\|_{L^2}^2
\lesssim (1+t)^{-\frac{6-q}{q}}+\varepsilon\|\nabla^3 u\|_{L^2}^2.
\end{aligned}
\end{equation}
Integrating by part and applying \eqref{smallness-condition1}-\eqref{2.7},
Holder and Sobolev inequalities, it arrives at
\begin{equation}\label{236}
\begin{aligned}
&\int \nabla^2 (-h(\varrho)[\mu\Delta u+(\mu+\lambda)\nabla {\rm div}u])\nabla^2 u dx\\
&\lesssim \left(\|\nabla h(\varrho)\|_{L^3}\|\nabla^2 u\|_{L^6}+\|h(\varrho)\|_{L^\infty}\|\nabla^3 u\|_{L^2}\right)
           \|\nabla^3 u\|_{L^2}\\
&\lesssim \delta_0 \|\nabla^3 u\|_{L^2}^2.
\end{aligned}
\end{equation}
It follows from the integration by part, \eqref{Decay-Pu-Guo}, \eqref{2.6}, \eqref{2.7},
Holder and Sobolev inequalities that
\begin{equation}\label{237}
\begin{aligned}
&\int \nabla^2 [-E(\varrho, \sigma) \nabla \varrho]\nabla^2 u dx\\
&\lesssim \|\nabla E(\varrho, \sigma)\|_{L^3}\|\nabla \varrho\|_{L^6}\|\nabla^3 u\|_{L^2}
          +\|E(\varrho, \sigma)\|_{L^\infty}\|\nabla^2 \varrho\|_{L^2}\|\nabla^3 u\|_{L^2}\\
&\lesssim \|\nabla(\varrho, \sigma)\|_{L^3}^2\|\nabla^2 \varrho\|_{L^2}^2
          +\|(\varrho, \sigma)\|_{L^\infty}^2\|\nabla^2 \varrho\|_{L^2}^2+\varepsilon\|\nabla^3 u\|_{L^2}^2\\
&\lesssim \|\nabla(\varrho, \sigma)\|_{L^3}^2\|\nabla^2 \varrho\|_{L^2}^2
          +\|\nabla(\varrho, \sigma)\|_{H^1}^2\|\nabla^2 \varrho \|_{L^2}^2+\varepsilon\|\nabla^3 u\|_{L^2}^2\\
&\lesssim \|\nabla(\varrho, \sigma)\|_{H^1}^2\|\nabla^2 \varrho \|_{L^2}^2
          +\varepsilon\|\nabla^3 u\|_{L^2}^2\\
&\lesssim (1+t)^{-\frac{3}{q}+\frac{1}{2}}(1+t)^{-\frac{3}{q}+\frac{1}{2}}
          +\varepsilon\|\nabla^3 u\|_{L^2}^2\\
&\lesssim (1+t)^{-\frac{6-q}{q}}+\varepsilon\|\nabla^3 u\|_{L^2}^2.\\
\end{aligned}
\end{equation}
Similarly, it is easy to deduce
\begin{equation}\label{238}
\int \nabla^2 [-F(\varrho, \sigma) \nabla \sigma]\nabla^2 u dx
\lesssim  \|\nabla(\varrho, \sigma)\|_{H^1}^2\|\nabla^2  \sigma \|_{L^2}^2
          +\varepsilon\|\nabla^3 u\|_{L^2}^2\\
\lesssim  (1+t)^{-\frac{6-q}{q}}+\varepsilon\|\nabla^3 u\|_{L^2}^2.
\end{equation}
By integration by part, \eqref{Decay-Pu-Guo}, \eqref{2.6}, \eqref{2.7},
Holder, Sobolev and Cauchy inequalities, we get
\begin{equation}\label{239}
\begin{aligned}
&\int \nabla^2 (g(\varrho)(B \cdot \nabla B))\nabla^2 u dx\\
&\lesssim (\|\nabla g(\varrho)\|_{L^6}\|B\|_{L^6}\|\nabla B\|_{L^6}
           +\|g(\varrho)\|_{L^\infty}\|\nabla B\|_{L^3}\|\nabla B\|_{L^6})\|\nabla^3 u\|_{L^2}\\
&\quad  +\|g(\varrho)\|_{L^\infty}\|B\|_{L^6}\|\nabla^2 B\|_{L^3}\|\nabla^3 u\|_{L^2}\\
&\lesssim \|\nabla^2 \varrho\|_{L^2}^2\|\nabla B\|_{L^2}^2\|\nabla^2 B\|_{L^2}^2
           +\|\nabla B\|_{L^3}^2\|\nabla^2 B\|_{L^2}^2\\
&\quad     +\|\nabla B\|_{L^2}^2\|\nabla^2 B\|_{L^3}^2
           +\varepsilon\|\nabla^3 u\|_{L^2}^2\\
&\lesssim \|\nabla B\|_{H^1}^2\|\nabla^2 B\|_{H^1}^2+\varepsilon\|\nabla^3 u\|_{L^2}^2\\
&\lesssim (1+t)^{-\frac{3}{q}+\frac{1}{2}}(1+t)^{-\frac{3}{q}+\frac{1}{2}}
          +\varepsilon\|\nabla^3 u\|_{L^2}^2\\
&\lesssim  (1+t)^{-\frac{6-q}{q}}+\varepsilon\|\nabla^3 u\|_{L^2}^2.\\
\end{aligned}
\end{equation}
In the same manner, we obtain
\begin{equation}\label{2310}
\frac{1}{2}\int \nabla^2 (-g(\varrho)\nabla (|\nabla B|^2)) \nabla^2 u dx
%&\lesssim \|\nabla B\|_{H^1}^2\|\nabla^2 B\|_{H^1}^2+\varepsilon\|\nabla^3 u\|_{L^2}^2\\
\lesssim  (1+t)^{-\frac{6-q}{q}}+\varepsilon\|\nabla^3 u\|_{L^2}^2.
\end{equation}
Following the idea as \eqref{235} and \eqref{236} respectively, it is easy to deduce
\begin{equation}\label{2311}
\int \nabla^2 (-u\cdot \nabla \sigma)\nabla^2 \sigma  dx
\lesssim \|\nabla u\|_{H^1}^2\|\nabla^2 \sigma\|_{H^1}^2+\varepsilon\|\nabla^3 \sigma\|_{L^2}^2
\lesssim (1+t)^{-\frac{6-q}{q}}+\varepsilon \|\nabla^3 \sigma\|_{L^2}^2,
\end{equation}
and
\begin{equation}\label{2312}
\kappa\int \nabla^2(-h(\varrho) \nabla^2 \sigma)\nabla^2 \sigma  dx
\lesssim \delta_0 \|\nabla^3 \sigma\|_{L^2}^2.
\end{equation}
The application of integration by part, \eqref{Decay-Pu-Guo},
\eqref{2.6}, \eqref{2.7}, Holder, Sobolev and Cauchy inequalities yields
\begin{equation}\label{2313}
\begin{aligned}
&\int \nabla(G(\varrho, \sigma) {\rm div} u)\nabla^3 \sigma \ dx\\
&\lesssim (\|\nabla G(\varrho, \sigma)\|_{L^3}\|\nabla u\|_{L^6}
           +\|G(\varrho, \sigma)\|_{L^\infty}\|\nabla^2 u\|_{L^2})\|\nabla^3 \sigma\|_{L^2}\\
&\lesssim \|\nabla(\varrho, \sigma)\|_{L^3}^2\|\nabla^2 u\|_{L^2}^2
           +\|\nabla (\varrho, \sigma)\|_{H^1}^2\|\nabla^2 u\|_{L^2}^2
           +\varepsilon \|\nabla^3 \sigma\|_{L^2}^2\\
&\lesssim \|\nabla(\varrho, \sigma)\|_{H^1}^2\|\nabla^2 u\|_{L^2}^2
          +\varepsilon \|\nabla^3 \sigma\|_{L^2}^2\\
&\lesssim (1+t)^{-\frac{6-q}{q}}+\varepsilon \|\nabla^3 \sigma\|_{L^2}^2.
\end{aligned}
\end{equation}
Integrating by part and applying \eqref{Decay-Pu-Guo}, \eqref{2.6}, \eqref{2.7},
Holder and Sobolev inequalities, we get
\begin{equation}\label{2314}
\begin{aligned}
&\int \nabla^2 (g(\varrho)\left[2\mu|D(u)|^2+\lambda({\rm div}u)^2\right])\nabla^2 \sigma dx\\
&\approx \int \nabla(g(\varrho) |\nabla u|^2)\nabla^3 \sigma  dx\\
&\lesssim (\|\nabla g(\varrho)\|_{L^6}\|\nabla u\|_{L^6}^2
           +\|g(\varrho)\|_{L^\infty}\|\nabla u\|_{L^3}\|\nabla^2 u\|_{L^6})
           \|\nabla^3 \sigma\|_{L^2}\\
&\lesssim \|\nabla^2 \varrho\|_{L^2}^2\|\nabla^2 u\|_{L^2}^2
           +\|\nabla u\|_{H^1}^2\|\nabla^3 u\|_{L^2}^2
           +\varepsilon \|\nabla^3 \sigma\|_{L^2}^2\\
&\lesssim \|\nabla (\varrho, u)\|_{H^1}^2\|\nabla^2 u\|_{H^1}^2
          +\varepsilon \|\nabla^3 \sigma\|_{L^2}^2\\
&\lesssim (1+t)^{-\frac{6-q}{q}}
          +\varepsilon \|\nabla^3 \sigma\|_{L^2}^2.\\
\end{aligned}
\end{equation}
In the same manner, it is easy to deduce
\begin{equation}\label{2315}
\int\nabla^2 (g(\varrho)(\nu |{\rm curl} B|^2)) \nabla^2 \sigma dx
%\lesssim \|\nabla (\varrho, B)\|_{H^1}^2\|\nabla^2 B\|_{H^1}^2
%          +\varepsilon \|\nabla^3 \sigma\|_{L^2}^2
\lesssim (1+t)^{-\frac{6-q}{q}}
          +\varepsilon \|\nabla^3 \sigma\|_{L^2}^2.
\end{equation}
Substituting \eqref{234}-\eqref{2315} into \eqref{233} and
applying the smallness of $\delta_0$ and $\varepsilon$,
it is easy to deduce
$$
\frac{d}{dt}\int (|\nabla^2 \varrho|^2+|\nabla^2 u|^2+|\nabla^2 \sigma|^2) dx
+\int (\mu |\nabla^3 u|^2+\kappa |\nabla^3 \sigma|^2) dx
\lesssim (1+t)^{-\frac{6-q}{q}}+\varepsilon\|\nabla^3 \varrho\|_{L^2}^2.
$$
Therefore, we complete the proof of the lemma.
\end{proof}

Furthermore, we establish the following differential inequality for the third order
spatial derivatives of solutions.

\begin{lemm}\label{lemma2.4}
Under all the assumptions in Theorem \ref{Theorem1.1}, then we have
\begin{equation}\label{241}
\frac{d}{dt}\|\nabla^3(\varrho, u, \sigma)\|_{L^2}^2
+(\mu \|\nabla^4 u\|_{L^2}^2+\kappa \|\nabla^4 \sigma\|_{L^2}^2)
\lesssim (1+t)^{-\frac{6-q}{q}}+\varepsilon\|\nabla^3 \varrho\|_{L^2}^2.
\end{equation}
\end{lemm}
\begin{proof}
Taking $k=3$ in \eqref{232}, it is easy to see that
\begin{equation}\label{242}
\begin{aligned}
&\frac{1}{2}\frac{d}{dt}\int (|\nabla^3 \varrho|^2+|\nabla^3 u|^2+|\nabla^3 \sigma|^2) dx
+\int (\mu |\nabla^4 u|^2+(\mu+\lambda)|\nabla^3 {\rm div}u|^2+\nu |\nabla^4 \sigma|^2) dx\\
&=\!\int \!\nabla^3 (-\varrho {\rm div}u-u\cdot \nabla \varrho) \nabla^3 \varrho dx
  +\!\int\! \nabla^3 (-u\cdot \nabla u-
                   h(\varrho)[\mu\Delta u+(\mu+\lambda)\nabla {\rm div}u]) \nabla^3 u dx\\
&\quad +\int \nabla^3 (-E(\varrho, \sigma)\nabla \varrho-F(\varrho, \sigma)\nabla \sigma
             +g(\varrho)[B\cdot \nabla B-\frac{1}{2}\nabla(|B|^2)]) \nabla^3 u dx\\
&\quad +\!\int\! \nabla^3 (-u\cdot \nabla \sigma- h(\varrho)(\kappa \Delta \sigma)
       -G(\varrho, \sigma){\rm div}u
      +g(\varrho)\left[2\mu|D(u)|^2+\lambda({\rm div}u)^2\right]) \nabla^3 \sigma dx\\
&\quad +\int\nabla^3 (g(\varrho)(\nu |{\rm curl} B|^2)) \nabla^3 \sigma dx.
\end{aligned}
\end{equation}
By virtue of decay rates \eqref{Decay-Pu-Guo}, \eqref{smallness-condition1},
Holder, Sobolev and Cauchy inequalities, we obtain
\begin{equation}\label{243}
\begin{aligned}
&\int \nabla^3 (-\varrho {\rm div}u)\nabla^3 \varrho dx\\
&\lesssim  \|\varrho\|_{L^\infty}\|\nabla^4 u\|_{L^2}\|\nabla^3 \varrho\|_{L^2}
           +\|\nabla \varrho\|_{L^3} \|\nabla^3 u\|_{L^6}\|\nabla^3 \varrho\|_{L^2}\\
&\quad \   +\|\nabla^2 \varrho\|_{L^3} \|\nabla^2 u\|_{L^6}\|\nabla^3 \varrho\|_{L^2}
           +\|\nabla u\|_{L^\infty} \|\nabla^3 \varrho\|_{L^2}^2\\
&\lesssim \|\nabla^2 \varrho\|_{H^1}^2 \|\nabla^3 u\|_{L^2}^2
           +(\varepsilon+\delta_0)(\|\nabla^3 \varrho\|_{L^2}^2+\|\nabla^4 u\|_{L^2}^2)\\
&\lesssim (1+t)^{-\frac{3}{q}+\frac{1}{2}}(1+t)^{-\frac{3}{q}+\frac{1}{2}}
           +(\varepsilon+\delta_0)(\|\nabla^3 \varrho\|_{L^2}^2+\|\nabla^4 u\|_{L^2}^2)\\
&\lesssim (1+t)^{-\frac{6-q}{q}}+(\varepsilon+\delta_0)(\|\nabla^3 \varrho\|_{L^2}^2+\|\nabla^4 u\|_{L^2}^2).
\end{aligned}
\end{equation}
Similarly, it is easy to deduce
\begin{equation}\label{244}
\int \nabla^3 (-u \cdot \nabla \varrho)\nabla^3 \varrho dx
\lesssim (1+t)^{-\frac{6-q}{q}}+(\varepsilon+\delta_0)\|\nabla^3 \varrho\|_{L^2}^2.
\end{equation}
With the help of decay rates \eqref{Decay-Pu-Guo}, \eqref{smallness-condition1},
Holder, Sobolev and Cauchy inequalities, we get
\begin{equation}\label{245}
\begin{aligned}
&\int \nabla^3 (-u\cdot \nabla u)\nabla^3 u \ dx\\
&\lesssim (\|\nabla u\|_{L^3}\|\nabla^2 u\|_{L^6}+\|u\|_{L^3}\|\nabla^3 u\|_{L^6})\|\nabla^4 u\|_{L^2}\\
&\lesssim \|\nabla u\|_{H^1}^2\|\nabla^3 u\|_{L^2}^2+(\varepsilon+\delta_0)\|\nabla^4 u\|_{L^2}^2\\
&\lesssim (1+t)^{-\frac{3}{q}+\frac{1}{2}}(1+t)^{-\frac{3}{q}+\frac{1}{2}}
          +(\varepsilon+\delta_0)\|\nabla^4 u\|_{L^2}^2\\
&\lesssim (1+t)^{-\frac{6-q}{q}}+(\varepsilon+\delta_0)\|\nabla^4 u\|_{L^2}^2.
\end{aligned}
\end{equation}
By virtue of the integration by part and applying decay rates \eqref{Decay-Pu-Guo},
\eqref{smallness-condition1}-\eqref{2.7}, Holder and Sobolev inequalities,  we deduce
\begin{equation}\label{246}
\begin{aligned}
&\int \nabla^3 (-h(\varrho)[\mu\Delta u+(\mu+\lambda)\nabla {\rm div}u])\nabla^3 u dx\\
&\approx \int \nabla^2 (h(\varrho) \nabla^2 u)\nabla^4 u \ dx\\
&\lesssim (\|\nabla \varrho\|_{L^6}^2\|\nabla^2 u\|_{L^6}
            +\|\nabla^2 \varrho\|_{L^3}\|\nabla^2 u\|_{L^6})\|\nabla^4 u\|_{L^2}\\
&\quad \ +(\|\nabla \varrho\|_{L^3}\|\nabla^3 u\|_{L^6}
               +\|h(\varrho)\|_{L^\infty}\|\nabla^4 u\|_{L^2})\|\nabla^4 u\|_{L^2}\\
&\lesssim \|\nabla^2 \varrho\|_{H^1}^2\|\nabla^3 u\|_{L^2}^2
          +(\delta_0+\varepsilon)\|\nabla^4 u\|_{L^2}^2\\
&\lesssim (1+t)^{-\frac{6-q}{q}}+(\varepsilon+\delta_0)\|\nabla^4 u\|_{L^2}^2.
\end{aligned}
\end{equation}
The application of decay rates \eqref{Decay-Pu-Guo},
\eqref{2.6}, \eqref{2.7}, Holder and Sobolev inequalities yields directly
\begin{equation}\label{247}
\begin{aligned}
&-\int \nabla^3 [-E(\varrho, \sigma) \nabla \varrho]\nabla^3 udx\\
&=\int [\nabla^2 E(\varrho, \sigma)\nabla \varrho+2\nabla E(\varrho, \sigma)\nabla^2 \varrho
        +E(\varrho, \sigma)\nabla^3 \varrho]\nabla^4 u  dx\\
&\lesssim (\|\nabla \varrho+\nabla \sigma\|_{L^6}^2\|\nabla \varrho\|_{L^6}
          +\|\nabla \varrho\|_{L^3}\|\nabla^2 \varrho+\nabla^2 \sigma\|_{L^6})\|\nabla^4 u\|_{L^2}\\
&\quad    +(\|\nabla \varrho+\nabla \sigma\|_{L^3}\|\nabla^2 \varrho\|_{L^6}
          +\|E(\varrho, \sigma)\|_{L^\infty}\|\nabla^3 \varrho\|_{L^2})\|\nabla^4 u\|_{L^2}\\
&\lesssim \|\nabla^2(\varrho, \sigma)\|_{L^2}^6
  +\|\nabla(\varrho, \sigma)\|_{L^3}^2\|\nabla^3(\varrho, \sigma)\|_{L^2}^2
          +\varepsilon\|\nabla^4 u\|_{L^2}^2\\
&\quad +\|\nabla (\varrho, \sigma)\|_{H^1}^2\|\nabla^3 \varrho\|_{L^2}^2\\
&\lesssim \|\nabla^2(\varrho, \sigma)\|_{L^2}^6+\|\nabla(\varrho, \sigma)\|_{H^1}^2\|\nabla^3(\varrho,\sigma)\|_{L^2}^2
          +\varepsilon\|\nabla^4 u\|_{L^2}^2\\
&\lesssim (1+t)^{-\frac{6-q}{q}}+\varepsilon\|\nabla^4 u\|_{L^2}^2.
\end{aligned}
\end{equation}
In the same manner, it is easy to deduce
\begin{equation}\label{248}
\int \nabla^3 [-F(\varrho, \sigma) \nabla \sigma]\nabla^3 udx
\lesssim (1+t)^{-\frac{6-q}{q}}+\varepsilon\|\nabla^4 u\|_{L^2}^2.
\end{equation}
Integrating by part and applying \eqref{Decay-Pu-Guo},
\eqref{2.6}, \eqref{2.7}, Holder and Sobolev inequalities, we get
\begin{equation}\label{249}
\begin{aligned}
&\int \nabla^3 (g(\varrho)(B \cdot \nabla B))\nabla^3 u dx\\
&=-\int (\nabla^2 g(\varrho)B\nabla B+2\nabla g(\varrho)\nabla(B\nabla B)+g(\varrho)\nabla^2(B\nabla B))\nabla^4 u dx\\
&\lesssim (\||\nabla^2 \varrho|+|\nabla \varrho|^2\|_{L^6}\|B\|_{L^6}\|\nabla B\|_{L^6}
     +\|\nabla \varrho\|_{L^6}\|\nabla B\|_{L^6}^2)\|\nabla^4 u\|_{L^2}\\
&\quad    \ +(\|\nabla \varrho\|_{L^6}\|B\|_{L^6}\|\nabla^2 B\|_{L^6}
            +\|g(\varrho)\|_{L^\infty}\|\nabla B\|_{L^6}\|\nabla^2 B\|_{L^3})\|\nabla^4 u\|_{L^2}\\
&\quad   \  +\|g(\varrho)\|_{L^\infty}\|B\|_{L^\infty}\|\nabla^3 B\|_{L^2}\|\nabla^4 u\|_{L^2}\\
&\lesssim \|\nabla B\|_{L^2}^2\|\nabla^2 B\|_{L^2}^2
     +\|\nabla^2 B\|_{L^2}^2\|\nabla^2 B\|_{H^1}^2
     +\varepsilon \|\nabla^4 u\|_{L^2}^2\\
&\lesssim \|\nabla B\|_{H^1}^2\|\nabla^2 B\|_{H^1}^2+\varepsilon \|\nabla^4 u\|_{L^2}^2\\
&\lesssim (1+t)^{-\frac{3}{q}+\frac{1}{2}}(1+t)^{-\frac{3}{q}+\frac{1}{2}}
          +\varepsilon\|\nabla^4 u\|_{L^2}^2\\
&\lesssim (1+t)^{-\frac{6-q}{q}}+\varepsilon\|\nabla^4 u\|_{L^2}^2.
\end{aligned}
\end{equation}
Similarly, it is easy to deduce
\begin{equation}\label{2410}
\frac{1}{2}\int \nabla^3 (-g(\varrho)\nabla (|\nabla B|^2))\nabla^3 u dx
\lesssim (1+t)^{-\frac{6-q}{q}}+\varepsilon\|\nabla^4 u\|_{L^2}^2.
\end{equation}
Following the idea as \eqref{245} and \eqref{246} respectively, it is easy to deduce
\begin{equation}\label{2411}
\begin{aligned}
\int \nabla^3(-u\cdot \nabla \sigma)\nabla^3 \sigma dx
&\lesssim \|\nabla (u, \sigma)\|_{H^1}^2\|\nabla^3 (u, \sigma)\|_{L^2}^2
          +(\varepsilon+\delta_0)\|\nabla^4 \sigma\|_{L^2}^2\\
&\lesssim (1+t)^{-\frac{6-q}{q}}
           +(\varepsilon+\delta_0)\|\nabla^4 \sigma\|_{L^2}^2,
\end{aligned}
\end{equation}
and
\begin{equation}\label{2412}
\begin{aligned}
&\kappa \int \nabla^3(-h(\varrho) \nabla^2 \sigma)\nabla^3 \sigma \ dx\\
&\lesssim \|\nabla^2 \varrho\|_{H^1}^2\|\nabla^3 \sigma\|_{L^2}^2
       +(\varepsilon+\delta)\|\nabla^4 \sigma\|_{L^2}^2\\
&\lesssim (1+t)^{-\frac{6-q}{q}}
           +(\varepsilon+\delta)\|\nabla^4 \sigma\|_{L^2}^2.
\end{aligned}
\end{equation}
Integrating by part and applying decay rates \eqref{Decay-Pu-Guo},
\eqref{2.6}, \eqref{2.7}, Holder and Sobolev inequalities, it arrives at
\begin{equation}\label{2413}
\begin{aligned}
&\int \nabla^3[-G(\varrho,\sigma) {\rm div} u]\nabla^3 \sigma dx\\
&=\int [\nabla^2 G(\varrho, \sigma){\rm div} u+2\nabla G(\varrho, \sigma)\nabla {\rm div} u
   +G(\varrho, \sigma)\nabla^2{\rm div} u]\nabla^4  \sigma dx\\
&\lesssim [\|\nabla \varrho+\nabla \sigma\|_{L^6}^2\|\nabla u\|_{L^6}
          +\|\nabla^2 \varrho+\nabla^2 \sigma\|_{L^3}\|\nabla u\|_{L^6}\\
&\quad \quad  +\|\nabla \varrho+\nabla \sigma\|_{L^3}\|\nabla^2 u\|_{L^6}
          +\|G(\varrho, \sigma)\|_{L^\infty}\|\nabla^3 u\|_{L^2}]\|\nabla^4 \sigma\|_{L^2}\\
&\lesssim \|\nabla^2(\varrho,u, \sigma)\|_{L^2}^6
          +\|\nabla (\varrho, \sigma)\|_{H^2}^2\|\nabla^2 u\|_{H^1}^2
          +\varepsilon\|\nabla^4 \sigma\|_{L^2}^2\\
&\quad    +\|\nabla (\varrho, \sigma)\|_{H^1}^2\|\nabla^3 u\|_{L^2}^2\\
&\lesssim \|\nabla^2(\varrho, u, \sigma)\|_{L^2}^6
          +\|\nabla(\varrho, \sigma)\|_{H^2}^2\|\nabla^2 u\|_{H^1}^2
          +\varepsilon\|\nabla^4 \sigma\|_{L^2}^2\\
&\lesssim (1+t)^{-\frac{6-q}{q}}+\varepsilon\|\nabla^4 \sigma\|_{L^2}^2.
\end{aligned}
\end{equation}
Integrating by part and applying decay rates \eqref{Decay-Pu-Guo},
\eqref{2.6}, \eqref{2.7}, Holder and Sobolev inequalities, we get
\begin{equation}\label{2414}
\begin{aligned}
&\int \nabla^3 (g(\varrho)\left[2\mu|D(u)|^2+\lambda({\rm div}u)^2\right])\nabla^3 \sigma dx\\
&\approx \int \nabla^2(g(\varrho) |\nabla u|^2)\nabla^4 \sigma dx\\
&=\int (\nabla^2 g(\varrho) |\nabla u|^2+4\nabla g(\varrho)\nabla u\nabla^2 u)\nabla^4 \sigma dx\\
&\quad +\int(2g(\varrho)\nabla^2 u\nabla^2 u+2g(\varrho)\nabla u \nabla^3 u)\nabla^4 \sigma dx\\
&\lesssim (\|\nabla \varrho \|_{L^\infty}\|\nabla \varrho\|_{L^6}\|\nabla u\|_{L^6}^2
           +\|\nabla^2 \varrho\|_{L^6}\|\nabla u\|_{L^6}^2)\|\nabla^4 \sigma\|_{L^2}\\
&\quad \ +(\|\nabla \varrho\|_{L^6}\|\nabla u\|_{L^6}\|\nabla^2 u\|_{L^6}
           +\|\nabla^2 u\|_{L^3}\|\nabla^2 u\|_{L^6})\|\nabla^4 \sigma\|_{L^2}\\
&\quad ~  +\|\nabla u\|_{L^\infty}\|\nabla^3 u\|_{L^2}\|\nabla^4 \sigma\|_{L^2}\\
&\lesssim \|\nabla^2 \varrho\|_{H^1}^2\|\nabla^2 u\|_{L^2}^2
           +\|\nabla^2 u\|_{H^1}^2\|\nabla^3 u\|_{L^2}^2
           +\varepsilon \|\nabla^4 \sigma\|_{L^2}^2\\
&\lesssim \|\nabla^2 (\varrho, u)\|_{H^1}^2\|\nabla^2 u\|_{H^1}^2
          +\varepsilon \|\nabla^4 \sigma\|_{L^2}^2\\
&\lesssim (1+t)^{-\frac{6-q}{q}}+\varepsilon\|\nabla^4 \sigma\|_{L^2}^2.
\end{aligned}
\end{equation}
In the same manner, it is easy to deduce
\begin{equation}\label{2415}
\begin{aligned}
\int\nabla^3 (g(\varrho)(\nu |{\rm curl} B|^2)) \nabla^3 \sigma dx
&\lesssim \|\nabla^2 (\varrho, B)\|_{H^1}^2\|\nabla^2 B\|_{H^1}^2
           +\varepsilon \|\nabla^4 \sigma\|_{L^2}^2\\
&\lesssim (1+t)^{-\frac{6-q}{q}}+\varepsilon\|\nabla^4 \sigma\|_{L^2}^2.
\end{aligned}
\end{equation}
Substituting \eqref{243}-\eqref{2415} into \eqref{242} and
applying the smallness of $\delta_0$ and $\varepsilon$,
it is easy to deduce
$$
\frac{d}{dt}\int (|\nabla^3 \varrho|^2+|\nabla^3 u|^2+|\nabla^3 \sigma|^2) dx
+\int (\mu |\nabla^4 u|^2+\kappa |\nabla^4 \sigma|^2) dx
\lesssim (1+t)^{-\frac{6-q}{q}}+\varepsilon\|\nabla^3 \varrho\|_{L^2}^2.
$$
Therefore, we complete the proof of the lemma.
\end{proof}

Furthermore,, we establish the dissipative estimates for the density.

\begin{lemm}\label{lemma2.5}
Under all the assumptions in Theorem \ref{Theorem1.1}, then we have
\begin{equation}\label{251}
\frac{d}{dt}\int \nabla^2 u\cdot \nabla^3 \varrho dx+\int|\nabla^3 \varrho|^2 dx
\le C_1 (\|\nabla^3 \sigma\|_{L^2}^2+\|\nabla^3 u\|_{H^1}^2)+C_1(1+t)^{-\frac{6-q}{q}}.
\end{equation}
\end{lemm}
\begin{proof}
Taking $\nabla^2-$th spatial derivatives on both hand sidse of $\eqref{eq1}_2$,
multiplying by $\nabla^2 \varrho$ and integrating over $\mathbb{R}^3$, then we get
\begin{equation}\label{252}
\int (\nabla^2 u_t \cdot \nabla^3 \varrho+|\nabla^3 \varrho|^2) dx
=\int [\mu \Delta \nabla^2 u+(\mu+\lambda)\nabla^3{\rm div}u-\nabla^3 \sigma
       +\nabla^2 S_2]\nabla^3 \varrho dx.
\end{equation}
In order to deal with the term $\int \nabla^2 u_t \cdot \nabla^3 \varrho dx$,
we turn to time derivatives of velocity
to the density and apply the transport equation $\eqref{eq1}_1$. More precisely, we have
\begin{equation}\label{253}
\begin{aligned}
\int \nabla^2 u_t \cdot \nabla^3 \varrho dx
&=\frac{d}{dt}\int \nabla^2 u \cdot \nabla^3 \varrho dx-\int \nabla^2 u \cdot \nabla^3 \varrho_t dx\\
&=\frac{d}{dt}\int \nabla^2 u \cdot \nabla^3 \varrho dx
            +\int \nabla^2 {\rm div} u \cdot \nabla^2 \varrho_t dx\\
&=\frac{d}{dt}\int \nabla^2 u \cdot \nabla^3 \varrho dx
   -\int \nabla^2 {\rm div} u \cdot \nabla^2({\rm div}u+\varrho {\rm div}u+u\nabla \varrho)dx\\
\end{aligned}
\end{equation}
Substituting \eqref{253} into \eqref{252}, it is easy to deduce
\begin{equation}\label{254}
\begin{aligned}
&\frac{d}{dt}\int \nabla^2 u \cdot \nabla^3 \varrho dx+\int |\nabla^3 \varrho|^2 dx\\
&=\int |\nabla^2 {\rm div} u|^2 dx+\int \nabla^2 {\rm div} u \cdot \nabla^2(\varrho {\rm div}u+u\nabla \varrho)dx\\
&\quad +\int \nabla^2 S_2 \cdot \nabla^3 \varrho dx
       +\int [\mu \Delta \nabla^2 u+(\mu+\lambda)\nabla^3{\rm div}u-\nabla \sigma]
       \nabla^3 \varrho dx.
\end{aligned}
\end{equation}
Integrating by part and applying decay rates \eqref{Decay-Pu-Guo},
Holder and Sobolev inequalities, we obtain
\begin{equation}\label{255}
\begin{aligned}
&\int \nabla^2 {\rm div} u \cdot \nabla^2(\varrho {\rm div}u+u\nabla \varrho)dx\\
&=-\int \nabla^3{\rm div}u \cdot \nabla(\varrho {\rm div}u+u\nabla \varrho)dx\\
&\lesssim \|\nabla^4 u\|_{L^2}(\|\nabla \varrho\|_{L^3}\|\nabla u\|_{L^6}
         +\|\varrho\|_{L^6}\|\nabla^2 u\|_{L^3}+\|u\|_{L^6}\|\nabla^2 \varrho\|_{L^3})\\
&\lesssim \|\nabla \varrho\|_{H^1}^2\|\nabla^2 u\|_{L^2}^2
         +\|\nabla \varrho\|_{L^2}^2\|\nabla^2 u\|_{H^1}^2
         +\|\nabla u\|_{L^2}^2\|\nabla^2 \varrho\|_{H^1}^2
         +\varepsilon \|\nabla^4 u\|_{L^2}^2\\
&\lesssim (1+t)^{-\frac{6-q}{q}}+\varepsilon \|\nabla^4 u\|_{L^2}^2.\\
\end{aligned}
\end{equation}
Following the idea as in Lemma \ref{lemma2.4}, we deduce immediately
\begin{equation}\label{256}
\int \nabla^2 S_2 \cdot \nabla^3 \varrho dx
\lesssim (1+t)^{-\frac{6-q}{q}}+\varepsilon\|\nabla^3 \varrho\|_{L^2}^2.
\end{equation}
On the other hand, it is easy to see that
\begin{equation}\label{257}
\int [\mu \Delta \nabla^2 u+(\mu+\lambda)\nabla^3{\rm div}u-\nabla^3 \sigma]\nabla^3 \varrho dx
\lesssim \|\nabla^3 \sigma\|_{L^2}^2+\|\nabla^4 u\|_{L^2}^2
         +\varepsilon\|\nabla^3 \varrho\|_{L^2}^2.
\end{equation}
Plugging \eqref{255}-\eqref{257} into \eqref{254}, we complete the proof of the lemma.
\end{proof}

The optimal decay rates for the second order spatial derivatives of global
classical solutions are stated in the following lemma.

\begin{lemm}\label{lemma2.6}
Under all the assumptions in Theorem \ref{Theorem1.1}, then the global classical solution
$(\varrho, u, \sigma)$ of Cauchy problem \eqref{eq1}-\eqref{eq4} has the
decay rates
\begin{equation}\label{261}
\|\nabla^2 \varrho (t)\|_{H^{1}}^2+\|\nabla^2 u (t)\|_{H^{1}}^2
+\|\nabla^2 \sigma (t)\|_{H^{1}}^2 \le
C(1+t)^{-3\left(\frac{1}{q}-\frac{1}{2}\right)-2}
\end{equation}
for all $t \ge T^*( T^*~\text{is~a~positive~constant~defined~below})$.
\end{lemm}
\begin{proof}
Adding \eqref{231} with \eqref{241}, it is easy to deduce
\begin{equation}\label{262}
\frac{d}{dt}\|\nabla^2(\varrho, u, \sigma)\|_{H^1}^2
+(\mu \|\nabla^3 u\|_{L^2}^2+\kappa \|\nabla^3 \sigma\|_{H^1}^2)
\le C_2 \delta_0\|\nabla^3 \varrho\|_{L^2}^2
    +C_2(1+t)^{-\frac{6-q}{q}}.
\end{equation}
Multiplying \eqref{251} by $\frac{2C_2 \delta_0}{C_1}$ and
adding with \eqref{262}, we obtain
\begin{equation}\label{263}
\frac{d}{dt} M_2^3(t)
+C_3(\|\nabla^3 \varrho\|_{L^2}^2+ \|\nabla^3 u\|_{H^1}^2+ \|\nabla^3 \sigma\|_{H^1}^2)
\le C_4(1+t)^{-\frac{6-q}{q}},
\end{equation}
where
$$
M_2^3(t)=\|\nabla^2(\varrho, u, \sigma)\|_{H^1}^2
+\frac{2 C_2 \delta_0}{C_1}\int \nabla^2 u\cdot \nabla^3 \varrho dx.
$$
Applying the Young inequality and the smallness of $\delta_0$,
we have the following equivalent relations
\begin{equation}\label{264}
C_5^{-1} \|\nabla^2(\varrho, u, \sigma)\|_{H^1}^2
\le M_2^3(t)
\le C_5 \|\nabla^2(\varrho, u, \sigma)\|_{H^1}^2.
\end{equation}
From the inequality \eqref{263}, it is easy to deduce
$$
\begin{aligned}
&\frac{d}{dt} M_2^3(t)
+\frac{C_3}{2}(\|\nabla^3 \varrho\|_{L^2}^2+\|\nabla^3 \varrho\|_{L^2}^2
               + \|\nabla^3 u\|_{H^1}^2+ \|\nabla^3 \sigma\|_{H^1}^2)\\
&\le \frac{d}{dt} M_2^3(t)
+\frac{C_3}{2}(2\|\nabla^3 \varrho\|_{L^2}^2
               + 2\|\nabla^3 u\|_{H^1}^2
               + 2\|\nabla^3 \sigma\|_{H^1}^2)\\
&\le C_4(1+t)^{-\frac{6-q}{q}},
\end{aligned}
$$
or equivalently,
\begin{equation}\label{265}
\frac{d}{dt} M_2^3(t)
+\frac{C_3}{2}(\|\nabla^3 \varrho\|_{L^2}^2+\|\nabla^3 \varrho\|_{L^2}^2
               + \|\nabla^3 u\|_{H^1}^2+ \|\nabla^3 \sigma\|_{H^1}^2)
\le C_4(1+t)^{-\frac{6-q}{q}}.
\end{equation}
Similar to \eqref{FSM}, we have
\begin{equation}\label{266}
\int |\nabla^3 u|^2 dx \ge \frac{R}{1+t}\int |\nabla^2 u|^2  dx
-\left(\frac{R}{1+t}\right)^2\int |\nabla u|^2dx,
\end{equation}
and
\begin{equation}\label{267}
\int |\nabla^4 u|^2 dx\ge \frac{R}{1+t}\int |\nabla^3 u|^2 dx
-\left(\frac{R}{1+t}\right)^2\int |\nabla^2 u|^2dx.
\end{equation}
Adding \eqref{266} with \eqref{267}, we obtain
\begin{equation}\label{268}
\|\nabla^3 u\|_{H^1}^2
\ge \frac{R}{1+t} \|\nabla^2 u\|_{H^1}^2
-\left(\frac{R}{1+t}\right)^2\|\nabla u\|_{H^1}^2.
\end{equation}
In the same manner, it is easy to deduce
\begin{equation}\label{269}
\|\nabla^3 \varrho\|_{L^2}^2
\ge \frac{R}{1+t}\|\nabla^2 \varrho\|_{L^2}^2
-\left(\frac{R}{1+t}\right)^2\|\nabla \varrho\|_{L^2}^2,
\end{equation}
and
\begin{equation}\label{2610}
\|\nabla^3 \sigma\|_{H^1}^2
\ge \frac{R}{1+t} \|\nabla^2 \sigma\|_{H^1}^2
-\left(\frac{R}{1+t}\right)^2\|\nabla \sigma\|_{H^1}^2.
\end{equation}
Combining \eqref{268}-\eqref{2610} with \eqref{265} and
applying the decay rates \eqref{Decay-Pu-Guo}, then we get
\begin{equation}\label{2611}
\begin{aligned}
&\frac{d}{dt} M_2^3(t)
+\frac{C_3}{2}\left[\frac{R}{1+t}(\|\nabla^2 \varrho\|_{L^2}^2+\|\nabla^2 u\|_{H^1}^2+\|\nabla^2 \sigma\|_{H^1}^2)+\|\nabla^3 \varrho\|_{L^2}^2 \right]\\
&\lesssim \left(\frac{R}{1+t}\right)^2
          (\|\nabla \varrho\|_{L^2}^2+\|\nabla u\|_{H^1}^2+\|\nabla \sigma\|_{H^1}^2)
          +(1+t)^{-\frac{6-q}{q}}\\
&\lesssim (1+t)^{-2}(1+t)^{-\left(\frac{3}{q}-\frac{1}{2}\right)}
          +(1+t)^{-\left(\frac{6}{q}-1\right)}\\
&\lesssim (1+t)^{-\left(\frac{3}{q}+\frac{3}{2}\right)}+(1+t)^{-\left(\frac{6}{q}-1\right)}\\
&\lesssim (1+t)^{-\left(\frac{3}{q}+\frac{3}{2}\right)},
\end{aligned}
\end{equation}
where we have used the fact \eqref{Large-Small-Relation}.
For some large time $t \ge R-1$, we have
$$
\frac{R}{1+t}\le 1,
$$
which implies
\begin{equation}\label{2612}
\frac{R}{1+t}\|\nabla^3 \varrho\|_{L^2}^2
\le \|\nabla^3 \varrho\|_{L^2}^2.
\end{equation}
Plugging \eqref{2612} into \eqref{2611}, it is easy to deduce
\begin{equation*}
\frac{d}{dt} M_2^3(t)
+\frac{C_3 R}{2(1+t)}(\|\nabla^2 \varrho\|_{H^1}^2+\|\nabla^2 u\|_{H^1}^2
                   +\|\nabla^2 \sigma\|_{H^1}^2)
\lesssim (1+t)^{-\left(\frac{3}{q}+\frac{3}{2}\right)},
\end{equation*}
which, together with the equivalent relation \eqref{264}, gives directly
\begin{equation}\label{2613}
\frac{d}{dt} M_2^3(t)+\frac{C_3 R}{2C_5(1+t)}M_2^3(t)
\lesssim (1+t)^{-\left(\frac{3}{q}+\frac{3}{2}\right)}.
\end{equation}
Choosing $R=\frac{2(3+q)C_5}{q C_3}$ in \eqref{2613}, then we have
\begin{equation}\label{2614}
\frac{d}{dt} M_2^3(t)+\frac{q+3}{q(1+t)}M_2^3(t)
\lesssim (1+t)^{-\left(\frac{3}{q}+\frac{3}{2}\right)}.
\end{equation}
Multiplying \eqref{2614} by $(1+t)^{\frac{3+q}{q}}$, it arrives at
\begin{equation}\label{2615}
\frac{d}{dt} \left[(1+t)^{\frac{3+q}{q}} M_2^3(t)\right] \lesssim (1+t)^{-\frac{1}{2}}.
\end{equation}
The integration of \eqref{2615} over $[0, t]$ yields
$$
M_2^3(t) \le (1+t)^{-\frac{3+q}{q}}(M_2^3(0)+C(1+t)^{\frac{1}{2}}),
$$
which, together with the equivalent relation \eqref{264}, gives
$$
\|\nabla^2 \varrho (t)\|_{H^{1}}^2+\|\nabla^2 u (t)\|_{H^{1}}^2
+\|\nabla^2 \sigma (t)\|_{H^{1}}^2
\le  C (1+t)^{-\left(\frac{3}{q}+\frac{1}{2}\right)},
$$
for all $t \ge T_*:=\frac{2(3+q)C_5}{q C_3}-1$.
Therefore, we complete the proof of the lemma.
\end{proof}

Finally, we establish optimal decay rates for the third order spatial
derivatives of magnetic field.

\begin{lemm}\label{lemma2.7}
Under all the assumptions in Theorem \ref{Theorem1.1}, then the magnetic field
has the following time decay rate for all $t \ge T^*$
\begin{equation}\label{271}
\|\nabla^3 B(t)\|_{L^2}^2\le C(1+t)^{-3\left(\frac{1}{q}-\frac{1}{2}\right)-3}.
\end{equation}
\end{lemm}
\begin{proof}
Combining the time decay rates \eqref{231}, \eqref{261} with \eqref{2311}, we get
\begin{equation}\label{272}
\begin{aligned}
&\frac{d}{dt}\int |\nabla^3 B|^2 dx+\nu\int |\nabla^4 B|^2 dx\\
&\lesssim \|\nabla^2 B\|_{H^1}^2\|\nabla u\|_{H^2}^2
           +\|\nabla B\|_{H^1}^2\|\nabla^3 u\|_{L^2}^2\\
&\lesssim (1+t)^{-3\left(\frac{1}{q}-\frac{1}{2}\right)-2}(1+t)^{-3\left(\frac{1}{q}-\frac{1}{2}\right)-1}\\
&\lesssim (1+t)^{-\frac{6}{q}}.
\end{aligned}
\end{equation}
Combining \eqref{2313}, \eqref{272} and the time decay rates \eqref{231}, we obtain
\begin{equation}\label{273}
\begin{aligned}
&\frac{d}{dt}\int |\nabla^3 B|^2 dx+\frac{R\nu}{1+t}\int |\nabla^3 B|^2 dx\\
&\lesssim \frac{R^2 \nu}{(1+t)^2}\int |\nabla^2 B|^2 dx+(1+t)^{-\frac{6}{q}}\\
&\lesssim (1+t)^{-\left(\frac{3}{q}+\frac{5}{2}\right)}+(1+t)^{-\frac{6}{q}}.
\end{aligned}
\end{equation}
By virtue of $q\in \left[1, \frac{6}{5}\right)$, it is easy to see that
$$
\frac{6}{q}-\left(\frac{3}{q}+\frac{5}{2}\right)
=\frac{5\left(\frac{6}{5}-q\right)}{2q} \ge 0.
$$
Then, we deduce from the inequality \eqref{273} that
\begin{equation}\label{274}
\frac{d}{dt}\int |\nabla^3 B|^2 dx+\frac{R\nu}{1+t}\int |\nabla^3 B|^2 dx
\lesssim (1+t)^{-\left(\frac{3}{q}+\frac{5}{2}\right)}.
\end{equation}
Choosing $R=\frac{3+2q}{q\nu}$ and multiplying \eqref{274} by $(1+t)^{\frac{3+2q}{q}}$,
then we have
\begin{equation}\label{275}
\frac{d}{dt}[(1+t)^{\frac{3+2q}{q}}\|\nabla^3 B\|_{L^2}^2] \lesssim (1+t)^{-\frac{1}{2}}.
\end{equation}
Integrating \eqref{275}  over $[0, t]$, it arrives at
$$
\|\nabla^3 B (t)\|_{L^2}^2
\le (1+t)^{-\left(\frac{3}{q}+2\right)}(\|\nabla^3 B_0\|_{L^2}^2+C(1+t)^{\frac{1}{2}}),
$$
which implies the following time decay rate
$$
\|\nabla^3 B (t)\|_{L^2}^2 \le C(1+t)^{-3\left(\frac{1}{q}+\frac{1}{2}\right)}.
$$
Therefore, we complete the proof of lemma.
\end{proof}

\emph{\bf{Proof of Theorem \ref{Theorem1.1}:}}
With the help of Lemma \ref{lemma2.2}, Lemma \ref{lemma2.6},
and Lemma \ref{lemma2.7}, we complete the proof of the Theorem \ref{Theorem1.1}.

\section{Proof of Theorem \ref{Theorem1.2}}\label{section3}

\quad In this section, we establish time decay rates for the mixed space-time derivatives of solutions.

\begin{lemm}\label{lemma3.1}
Under the assumptions in Theorem \ref{Theorem1.1}, the global solution $(\varrho, u, \sigma, B)$
of problem
\eqref{eq1}-\eqref{eq4} has the time decay rates
$$
\begin{aligned}
&\|\nabla^k \varrho_t (t)\|_{H^{2-k}}
+\|\nabla^k u_t(t)\|_{L^2}
+\|\nabla^k \sigma_t(t)\|_{L^2}
\le C(1+t)^{-\frac{3}{2}\left(\frac{1}{q}-\frac{1}{2}\right)-\frac{k+1}{2}},\\
&\|\nabla^k B_t(t)\|_{L^2} \le C(1+t)^{-\frac{3}{2}\left(\frac{1}{q}-\frac{1}{2}\right)-\frac{k+2}{2}},\\
\end{aligned}
$$
where $k=0,1$.
\end{lemm}
\begin{proof}
First of all, applying \eqref{eq1}$_1$, \eqref{Decay1}, Holder and Sobolev inequalities, it arrives at
\begin{equation}\label{312}
\begin{aligned}
\|\varrho_t\|_{L^2}^2
&=\|-{\rm div}u-\varrho {\rm div}u-u \cdot \nabla \varrho\|_{L^2}^2\\
&\lesssim \|\nabla u\|_{L^2}^2+\|\varrho\|_{L^6}^2\|\nabla u\|_{L^3}^2
          +\|u\|_{L^6}^2\|\nabla \varrho\|_{L^3}^2\\
&\lesssim (1+t)^{-\frac{3}{q}+\frac{1}{2}}
          +(1+t)^{-\frac{3}{q}+\frac{1}{2}}(1+t)^{-\frac{3}{q}+\frac{1}{2}}\\
&\lesssim (1+t)^{-\frac{3}{q}+\frac{1}{2}}.
\end{aligned}
\end{equation}
Similarly, we obtain
\begin{equation}\label{313}
\begin{aligned}
\|\nabla \varrho_t\|_{L^2}^2
&=\|-\nabla{\rm div}u- \nabla(\varrho {\rm div}u-u \cdot \nabla \varrho)\|_{L^2}^2\\
&\lesssim \|\nabla^2 u\|_{L^2}^2+\|\nabla \varrho\|_{L^3}^2\|\nabla u\|_{L^6}^2
          +\|\varrho\|_{L^3}^2\|\nabla^2 u\|_{L^6}^2\\
&\quad \  +\|\nabla u\|_{L^6}^2\|\nabla \varrho\|_{L^3}^2
          +\|u\|_{L^6}^2\|\nabla^2 \varrho\|_{L^3}^2\\
&\lesssim \| \nabla^2 \varrho\|_{H^1}^2+\| \nabla^2 u\|_{H^1}^2\\
&\lesssim (1+t)^{-\frac{3}{q}-\frac{1}{2}},
\end{aligned}
\end{equation}
and
\begin{equation}\label{314}
\begin{aligned}
\|\nabla^2 \varrho_t\|_{L^2}^2
&=\|\nabla^2(-{\rm div}u-\varrho{\rm div}u-u \cdot \nabla \varrho)\|_{L^2}^2\\
&\lesssim \|\nabla^3 u\|_{L^2}^2+\|\nabla^2 \varrho\|_{L^3}^2\|\nabla u\|_{L^6}^2
          +\|\varrho\|_{L^\infty}^2\|\nabla^3 u\|_{L^2}^2\\
&\quad \  +\|\nabla \varrho\|_{L^6}^2\|\nabla^2 u\|_{L^3}^2
          +\|u\|_{L^\infty}^2\|\nabla^3 \varrho\|_{L^2}^2\\
&\lesssim \|\nabla^3 u\|_{L^2}^2+\|\nabla^2 \varrho\|_{H^1}^2\|\nabla^2 u\|_{H^1}^2
          +\|\nabla u\|_{H^1}^2 \|\nabla^3 \varrho\|_{L^2}^2\\
&\lesssim \|\nabla^3 u\|_{L^2}^2+\|\nabla^2 \varrho\|_{H^1}^2\\
&\lesssim (1+t)^{-\frac{3}{q}-\frac{1}{2}}.
\end{aligned}
\end{equation}
Combining \eqref{312}-\eqref{314}, then we have the time decay rates
\begin{equation}\label{315}
\|\nabla^k \varrho_t(t)\|_{H^{2-k}}^2 \le C(1+t)^{-3\left(\frac{1}{q}-\frac{1}{2}\right)-(k+1)},
\end{equation}
where $k=0, 1$.
Secondly, in view of the equation $\eqref{eq1}_2$, \eqref{Decay1}, \eqref{2.6},
\eqref{2.7}, Holder and Sobolev inequalities, it is easy to deduce
\begin{equation}\label{316}
\begin{aligned}
\|u_t\|_{L^2}^2
&=\|\mu \Delta u+(\mu+\lambda)\nabla {\rm div} u-\nabla \varrho-\nabla \sigma\|_{L^2}^2\\
&\quad +\|-u\cdot \nabla u-h(\varrho)[\mu \Delta u+(\mu+\lambda)\nabla {\rm div} u]
          -E(\varrho, \sigma)\nabla \varrho\|_{L^2}^2\\
&\quad +\|-F(\varrho, \sigma)\nabla \sigma
          +g(\varrho)[B\cdot \nabla B-\frac{1}{2}\nabla (|B|^2)]\|_{L^2}^2\\
&\lesssim \|\nabla^2 u\|_{L^2}^2+\|\nabla(\varrho, \sigma)\|_{L^2}^2
          +\|u\|_{L^6}^2\|\nabla u\|_{L^3}^2+\|h(\varrho)\|_{L^\infty}^2\|\nabla^2 u\|_{L^2}^2\\
&\quad \ +\|(E(\varrho, \sigma), F(\varrho, \sigma))\|_{L^\infty}^2\|\nabla(\varrho, \sigma)\|_{L^2}^2
         +\|g(\varrho)\|_{L^\infty}^2\|B\|_{L^6}^2\|\nabla B\|_{L^3}^2\\
&\lesssim \|\nabla (u, B)\|_{H^1}^2+\|\nabla(\varrho, \sigma)\|_{L^2}^2\\
&\lesssim (1+t)^{-\frac{3}{q}+\frac{1}{2}}.
\end{aligned}
\end{equation}
Similarly, it is easy to deduce
\begin{equation}\label{317}
\begin{aligned}
\|\nabla u_t\|_{L^2}^2
&=\|\nabla(\mu \Delta u+(\mu+\lambda)\nabla {\rm div} u-\nabla \varrho-\nabla \sigma)\|_{L^2}^2
    +\|\nabla S_2\|_{L^2}^2\\
&\lesssim \|\nabla^3 u\|_{L^2}^2+\|\nabla^2 (\varrho, \sigma)\|_{L^2}^2
           +\|\nabla u\|_{H^1}^2\|\nabla^2 u\|_{H^1}^2\\
&\quad     +\|\nabla(\varrho, \sigma)\|_{H^1}^2\|\nabla^2 (\varrho,\sigma) \|_{L^2}^2
           +\|\nabla B\|_{H^1}^2\|\nabla^2 B\|_{H^1}^2\\
&\lesssim \|\nabla^2 (u, B)\|_{H^1}^2+\|\nabla^2 (\varrho, \sigma)\|_{L^2}^2\\
&\lesssim (1+t)^{-\frac{3}{q}-\frac{1}{2}}.
\end{aligned}
\end{equation}
Combining \eqref{316}-\eqref{317}, then we have the time decay rates
\begin{equation}\label{318}
\|\nabla^k u_t(t)\|_{L^{2}}^2 \le C(1+t)^{-3\left(\frac{1}{q}-\frac{1}{2}\right)-(k+1)},
\end{equation}
where $k=0, 1$.
Furthermore, it is easy to deduce
\begin{equation}\label{319}
\begin{aligned}
\|\sigma_t\|_{L^2}^2
&=\|\kappa \Delta \sigma-{\rm div}u-u\cdot \nabla \sigma
    - h(\varrho)(\kappa \Delta \sigma)-G(\varrho, \sigma){\rm div}u\|_{L^2}^2\\
&\quad +\|g(\varrho)\left[2\mu|D(u)|^2
       +\lambda({\rm div}u)^2\right]+g(\varrho)(\nu |{\rm curl} B|^2)\|_{L^2}^2\\
&\lesssim \|\nabla^2 \sigma\|_{L^2}^2+\|\nabla u\|_{L^2}^2
          +\|u\|_{L^6}^2\|\nabla \sigma\|_{L^3}^2
          +\|h(\varrho)\|_{L^\infty}^2\|\nabla^2 \sigma\|_{L^2}^2\\
&\quad +\|G(\varrho, \sigma)\|_{L^\infty}^2\|\nabla u\|_{L^2}^2
       +\|g(\varrho)\|_{L^\infty}^2\|\nabla u\|_{L^3}^2\|\nabla u\|_{L^6}^2\\
&\quad +\|g(\varrho)\|_{L^\infty}^2\|\nabla B\|_{L^3}^2\|\nabla B\|_{L^6}^2\\
&\lesssim \|\nabla^2 \sigma\|_{L^2}^2+\|\nabla (u, B)\|_{H^1}^2\\
&\lesssim (1+t)^{-\frac{3}{q}+\frac{1}{2}}.
\end{aligned}
\end{equation}
Similarly, it is easy to deduce
\begin{equation}\label{3110}
\begin{aligned}
\|\nabla \sigma_t\|_{L^2}^2
&=\|\nabla(\kappa \Delta \sigma-{\rm div}u)\|_{L^2}^2+\|\nabla S_3\|_{L^2}^2\\
&\lesssim \|\nabla^3 \sigma\|_{L^2}^2+\|\nabla^2 u\|_{L^2}^2
           +\|\nabla u\|_{H^1}^2\|\nabla^2 \sigma\|_{H^1}^2\\
&\quad     +\|\nabla(\varrho, \sigma)\|_{H^1}^2\|\nabla^2 u\|_{L^2}^2
           +\|\nabla (\varrho, u)\|_{H^1}^2\|\nabla^2 u\|_{H^1}^2\\
&\quad     +\|\nabla (\varrho, B)\|_{H^1}^2\|\nabla^2 B\|_{H^1}^2\\
&\lesssim \|\nabla^2 (\sigma, u, B)\|_{H^1}^2\\
&\lesssim (1+t)^{-\frac{3}{q}-\frac{1}{2}}.
\end{aligned}
\end{equation}
In view of \eqref{319}-\eqref{3110}, then we have the time decay rates
\begin{equation}\label{3112}
\|\nabla^k \sigma_t(t)\|_{L^{2}}^2 \le C(1+t)^{-3\left(\frac{1}{q}-\frac{1}{2}\right)-(k+1)},
\end{equation}
where $k=0, 1$.
Finally, it follows from the $\eqref{eq1}_3$, Holder and Sobolev inequalities that
\begin{equation}\label{3113}
\begin{aligned}
\|B_t\|_{L^2}^2
&=\|\nu \Delta B-u \cdot \nabla B+B\cdot \nabla u-B {\rm div}u\|_{L^2}^2\\
&\lesssim \|\nabla^2 B\|_{L^2}^2+\|u\|_{L^6}^2\|\nabla B\|_{L^3}^2
          +\|B\|_{L^6}^2\|\nabla u\|_{L^3}^2\\
&\lesssim \|\nabla^2 B\|_{L^2}^2
          +\|\nabla(u, B)\|_{L^2}^2\|\nabla (u, B)\|_{H^1}^2\\
&\lesssim (1+t)^{-\frac{3}{q}-\frac{1}{2}}+(1+t)^{-\frac{3}{q}+\frac{1}{2}}(1+t)^{-\frac{3}{q}+\frac{1}{2}}\\
&\lesssim (1+t)^{-\frac{3}{q}-\frac{1}{2}}.
\end{aligned}
\end{equation}
Similarly, it follows from $\eqref{eq1}_3$, \eqref{233}-\eqref{235} that
\begin{equation}\label{3114}
\begin{aligned}
\|\nabla B_t\|_{L^2}^2
&=\|\nabla (\nu \Delta B-u \cdot \nabla B+B\cdot \nabla u-B {\rm div}u)\|_{L^2}^2\\
&\lesssim \|\nabla^3 B\|_{L^2}^2+\|\nabla u\|_{H^2}^2\|\nabla B\|_{H^2}^2\\
&\lesssim (1+t)^{-\frac{3}{q}-\frac{3}{2}}+(1+t)^{-\frac{6}{p}+1}\\
&\lesssim (1+t)^{-\frac{3}{q}-\frac{3}{2}}.
\end{aligned}
\end{equation}
By view of \eqref{3113} and \eqref{3114}, it is easy to obtain
\begin{equation}\label{3115}
\|\nabla^k B_t(t)\|_{L^2}^2 \le C(1+t)^{-3\left(\frac{1}{q}-\frac{1}{2}\right)-(k+2)},
\end{equation}
where $k=0,1$.
Therefore, we complete the proof of the lemma.
\end{proof}

\emph{\bf{Proof of Theorem \ref{Theorem1.2}:}} With the help of Lemma \ref{lemma3.1},
we complete the proof of the Theorem \ref{Theorem1.2}.

\section*{Acknowledgements}

Zheng-an Yao's research is supported in part by NNSFC(Grant No.11271381)
and China 973 Program(Grant No. 2011CB808002).
Qiang Tao's research is supported by the NSF of China under grant 11171060 and 11301345,
by Guangdong Natural Science Foundation under grant 2014A030310074.

\phantomsection

\end{document}